\documentclass[10pt]{amsart}
\usepackage{graphicx}
\usepackage{hyperref}
\usepackage{amssymb,amsmath,amscd,amsthm,epsfig, tikz}
\usepackage[enableskew, vcentermath]{youngtab}

\newcommand{\ZZ}{\mathbb{Z}}

\newcommand{\CC}{\mathbb{C}}

\newcommand{\SYT}{{\rm {SYT}}}
\newcommand{\la}{\lambda}

\newcommand{\BBB}{{\mathcal{B}}}
\newcommand{\CCC}{{\mathcal{C}}}

\newcommand{\Des}{\operatorname{Des}}

\newcommand{\sign}{\operatorname{sign}}
\newcommand{\fin}{\operatorname{maxout}}

\newcommand{\ord}{\operatorname{ord}}
\newcommand{\Comp}{\operatorname{Comp}}

\newcommand{\then}{\Longrightarrow}

\newcommand{\maj}{\operatorname{maj}}
\newcommand{\x}{{\mathbf x}}

\newtheorem{theorem}{Theorem}[section]
\newtheorem{corollary}[theorem]{Corollary}
\newtheorem{proposition}[theorem]{Proposition}

\newtheorem{lemma}[theorem]{Lemma}

\newtheoremstyle{defn}{1.2ex}{1.2ex}{}{}{}{.}{.5em}%
{\textbf{\thmname{#1}\thmnumber{ #2}}\thmnote{\emph{ (#3)}}}
\theoremstyle{defn}
\newtheorem{defn}[theorem]{Definition}

\newtheorem{example}[theorem]{Example}
\newtheorem{remark}[theorem]{Remark}
\newtheorem{question}[theorem]{Question}
\newtheorem{fact}[theorem]{Fact}
\newtheorem{observation}[theorem]{Observation}
\newtheorem{claim}[theorem]{Claim}

\newcommand{\todo}[1]{\vspace{2 mm}\par\noindent
   \marginpar{\textsc{ToDo}}\framebox{\begin{minipage}[c]{0.95 \textwidth}
   \tt #1 \end{minipage}}\vspace{2 mm}\par}

\numberwithin{figure}{section}
\numberwithin{table}{section}

\begin{document}
\title{Matrices, Characters and Descents}

\author{Ron M.\ Adin}
\address{Department of Mathematics\\
Bar-Ilan University\\
52900 Ramat-Gan\\
Israel} \email{radin@math.biu.ac.il}
\thanks{Both authors were partially supported by Internal Research Grants
from the Office of the Rector, Bar-Ilan University}

\author{Yuval Roichman}
\address{Department of Mathematics\\
Bar-Ilan University\\
52900 Ramat-Gan\\
Israel} \email{yuvalr@math.biu.ac.il}

\begin{abstract}
A new family of asymmetric matrices of Walsh-Hadamard type is introduced.
We study their properties and, in particular, compute their determinants and
discuss their eigenvalues.
The invertibility of these matrices implies that certain character formulas are invertible,
yielding expressions for the cardinalities of sets of combinatorial objects with prescribed
descent sets in terms of character values of the symmetric group.
\end{abstract}

\keywords{Walsh-Hadamard matrices, symmetric group, character formulas, descents}


\date{submitted: Sep.\ 11, '13; revised: Nov.\ 18, '14}

\maketitle


\section{Introduction}\label{section:intro}


\medskip
\subsection{Overview}\ 
\medskip

Many character formulas involve the descent set of
a permutation or of a standard Young tableau.
We propose here a general setting for such formulas, involving
a new family of asymmetric matrices of Walsh-Hadamard type.
These matrices turn out to have fascinating properties,
some of which are studied here using
a transformation based on M\"obius inversion.
These include the evaluation of numerical attributes such as
the determinants of the matrices and the entries of related matrices.
An explicit combinatorial description of the eigenvalues,
conjectured in an earlier version of this paper, has recently been proved by G.\ Alon.


The inverse matrices lead to
formulas expressing the
cardinalities of sets of combinatorial objects with prescribed ``descent sets''
(extending the familiar notion of descent set of a permutation)
in terms of character values of the symmetric group.
Examples of such objects include 
permutations of fixed length, involutions, standard Young
tableaux, and more. It also follows that certain statements about
permutation statistics have equivalent formulations in character
theory. For example, the fundamental equi-distribution Theorem of
Foata and Sch\"utzenberger, independently proved by Garsia and
Gessel, is equivalent to a theorem of Lusztig and Stanley in
invariant theory.




\medskip
\subsection{Outline and main results}\ 
\medskip

Beyond the preliminaries in Section~\ref{section:preliminaries}, this paper consists of
a linear algebraic part (Sections~\ref{section:matrices}--\ref{section:matrix_entries})
and a character theoretic part (Sections~\ref{section:char_fine} and
\ref{section:applications}--\ref{section:fine_sets}).
Let us list the main results in each part.

\medskip
\subsubsection{Matrices}\ 
\medskip

Recall the well known {\em Walsh-Hadamard (Sylvester)} matrices,
defined by the recursion
\[
H_n = \left(\begin{array}{cc}
H_{n-1} & H_{n-1} \\
H_{n-1} & -H_{n-1}
\end{array}\right)
\qquad(n \ge 1)
\]
with $H_0 = (1)$.

\begin{defn}\label{d.AB_recursion}
Define, recursively,
\[
A_n = \left(\begin{array}{cc}
A_{n-1} & A_{n-1} \\
A_{n-1} & -B_{n-1}
\end{array}\right)
\qquad(n \ge 1)
\]
with $A_0 = (1)$, and
\[
B_n = \left(\begin{array}{cc}
A_{n-1} & A_{n-1} \\
0 & -B_{n-1}
\end{array}\right)
\qquad(n \ge 1)
\]
with $B_0 = (1)$.
\end{defn}

Each of the matrices $A_n$ and $B_n$ may be obtained from the
corresponding Walsh-Hadamard matrix $H_n$, all the entries of
which are $\pm 1$, by replacing some of the entries by $0$.

\medskip

\begin{theorem}\label{main1} (Theorem~\ref{t.An-determinant}
below)\\
\[
\det(A_n) = (n+1) \cdot \prod_{k=1}^{n} k^{2^{n-1-k} (n+4-k)}
\qquad(n \ge 2)
\]
while $\det(A_0) = 1$ and $\det(A_1) = -2$, and
\[
\det(B_n) = \prod_{k=1}^{n} k^{2^{n-1-k} (n+2-k)} \qquad(n \ge 2)
\]
while $\det(B_0) = 1$ and $\det(B_1) = -1$.
\end{theorem}

It follows that $A_n$ and $B_n$ are invertible for all $n \ge 0$.
The invertibility of $A_n$ will be applied to provide
a generalized setting for the study of character formulas.


In Sections~\ref{section:matrices}--\ref{section:matrix_entries}
we shall study and discuss various properties of these matrices.
These sections may be read as an independent ``linear algebraic'' unit,
not involving the character theoretic motivation and applications
discussed elsewhere in the paper.

\medskip
\subsubsection{Fine sets}\ 
\medskip

Recall the {\em descent set} of a permutation $\pi\in S_n$, defined by
\[
\Des(\pi):=\{i:\ \pi(i)>\pi(i+1)\}.
\]


A sequence $(a_1, \ldots, a_n)$ of distinct positive integers is {\em unimodal}
if there exists $1 \le m\le n$ such that
\[
a_1 > a_2 > \ldots > a_m < a_{m+1} < \ldots < a_n.
\]
Let $\mu=(\mu_1,\dots,\mu_t)$ be a composition of $n$.
A sequence of $n$ positive integers is {\em $\mu$-unimodal} if
the first $\mu_1$ integers form a unimodal sequence,
the next $\mu_2$ integers form a unimodal sequence, and so on.
A permutation $\pi \in S_n$ is {\em $\mu$-unimodal} if
the sequence $(\pi(1), \ldots, \pi(n))$ is $\mu$-unimodal.
Denote by $U_\mu$ the set of all $\mu$-unimodal permutations
in $S_n$.


For a positive integer $n$ denote the set of all compositions of $n$ by $\Comp(n)$.
For a composition $\mu=(\mu_1, \ldots, \mu_t)$ set
$S(\mu) := \{\mu_1, \mu_1+\mu_2, \ldots\}$.

\medskip

Motivated by several examples (to be described in Section~\ref{section:char_fine})
we define the following concept.

\begin{defn}
A subset $\BBB\subseteq S_n$ is a {\em fine set} for a complex
$S_n$-representation $\rho$ if, for each composition $\mu$ of $n$,
the character value of $\rho$ at a conjugacy class of cycle type $\mu$ satisfies
\[
\chi^\rho(\mu) = \sum\limits_{\pi \in \BBB\cap U_\mu}
(-1)^{|\Des(\pi) \setminus S(\mu)|}.
\]
\end{defn}
Examples of fine sets include
\begin{itemize}
\item
permutations of fixed Coxeter length (equivalently, fixed inversion number);
\item
conjugacy classes and their unions
(e.g., involutions or permutations with fixed cycle number);
\item
Knuth classes and their unions
(e.g., inverse descent classes or permutations with fixed descent number).
\end{itemize}
Definitions of these sets and further examples will be given in
the sequel.

\medskip

The invertibility of the matrix $A_n$ will be 
applied to prove the following.

\begin{theorem}\label{main2}(Theorem~\ref{t.main} below)\\
If $\BBB$ is a fine set for an $S_n$-representation $\rho$ then
the character values of $\rho$ uniquely determine the distribution
of descent sets over $\BBB$.
\end{theorem}

An explicit formula for the distribution of the descent set over
$\BBB$ in terms of the character values will be given in
Theorem~\ref{t.main} below.

\bigskip

Theorem~\ref{main2} will then be applied to provide various
equivalent descriptions of fine sets.

\begin{theorem}\label{main3}
For every subset $\BBB\subseteq S_n$, the following are equivalent.
\begin{itemize}
\item[$(i)$]
$\BBB$ is fine; in other words,
the function $\chi^\BBB : \Comp(n) \longrightarrow \ZZ$ defined by
\[
\chi^\BBB(\mu):=\sum\limits_{\pi\in \BBB\cap U_\mu}
(-1)^{|\Des(\pi)\setminus S(\mu)|}
\]
is an $S_n$-character, i.e., it does not depend on the order of parts in $\mu$
and is a linear combination, with nonnegative integer coefficients,
of the irreducible characters of $S_n$.
\item[$(ii)$]
There exists a set partition $\BBB = \BBB_1 \dot{\cup} \ldots \dot{\cup} \,\BBB_m$
such that, for each $1 \le i \le m$, there exists a descent-preserving bijection
from $\BBB_i$ to the set of all standard Young tableaux of
a suitable shape $\la^{(i)} \vdash n$.
\item[$(iii)$]
The quasi-symmetric function $F_\BBB(\x)$, defined in
Subsection~\ref{subsection:quasi}, is symmetric and Schur
positive.
\item[$(iv)$]
There exists a basis $\{C_b : b\in \BBB\}$ for an $S_n$-representation space $V$
such that the linear action of every simple reflection $s_i$ has the form
\[
s_i (C_b) =
\begin{cases}
-C_b, &\text{\rm if } i \in \Des (b); \\
C_b + \sum_{b' \in \BBB \text{ \rm s.t.\ } i \in\Des(b')}
a_i(b,b') C_{b'}, &\text{\rm otherwise,}
\end{cases}
\]
for suitable coefficients $a_i(b,b')$.
\end{itemize}
\end{theorem}


See Propositions~\ref{t.criterion1} and~\ref{t.criterion_q} and
Theorem~\ref{t.condition-fine}.

\medskip

More examples of fine sets, generalizations and applications will
be given in Sections~\ref{section:char_fine} and
\ref{section:applications}--\ref{section:fine_sets} below.

\tableofcontents


%

\section{Preliminaries and notation}\label{section:preliminaries}

\medskip
\subsection{Intervals, runs, compositions and partitions}\label{subsection:prelim_compositions}\ 
\medskip

For nonnegative integers $m, n$ denote the {\em interval}
\[
[m,n]:= \begin{cases}
\{m, m+1, \ldots, n\}, & \text{if $m \le n$;}\\
\emptyset, & \text{otherwise.}
\end{cases}
\]
Denote also $[n] := [1, n] = \{1, \ldots, n\}$ ($[0] = \emptyset$).

\begin{defn}
A {\em prefix} of an interval $[m + 1, m + \ell]$ (with $\ell \ge 0$) is
an interval $[m + 1, m + p]$ for some $0 \le p \le \ell$.
\end{defn}

\begin{defn}\label{def.runs}
Let $I$ be a finite set of integers.
The {\em runs} in $I$ are its maximal subsets which are nonempty intervals.
Thus, if $I_1, \ldots, I_t$ is the sequence of runs in $I$, then
$I$ is their disjoint union and each $I_k$ has the form
$\{ m_k+1, m_k+2, \ldots, m_k+\ell_k \}$ with $\ell_k\ge 1$ $(\forall k)$
and $m_1 < m_1+\ell_1 < m_2 < m_2+\ell_2 < \ldots < m_t < m_t+\ell_t$.
In particular, $|I| = \ell_1 + \ldots + \ell_t$.
\end{defn}

\begin{example}
If $I = \{1,2,4,5,6, 8, 10\}$ then
$I_1 = \{1,2\}$, $I_2 = \{4,5,6\}$, $I_3 = \{8\}$ and $I_4 = \{10\}$.
\end{example}

A {\em composition} of a positive integer $n$ is a vector
$\mu=(\mu_1, \ldots, \mu_t)$ of positive integers such that
$\mu_1+\cdots+\mu_t=n$.
A {\em partition} of $n$ is a composition with
weakly decreasing entries $\mu_1 \ge \ldots \ge \mu_t > 0$.
The {\em underlying partition} of a composition is obtained by
reordering the entries in weakly decreasing order.


For each composition $\mu = (\mu_1, \ldots, \mu_t)$ of $n$, define
the corresponding set of partial sums
\[
S(\mu) := 
\{ \mu_1, \mu_1+\mu_2, \ldots, \mu_1 + \ldots + \mu_t = n\} \subseteq [n],
\]
as well as its complement
\[
I(\mu) := [n] \setminus S(\mu) \subseteq [n-1].
\]
For example, the composition $\mu=(3,4,2,5)$ of $n = 14$ has
$S(\mu) = \{3,7,9,14\}$ and $I(\mu) = \{1,2,4,5,6,8,10,11,12,13\}$.

The correspondence $\mu \longleftrightarrow I(\mu)$ is  a bijection between
the set of all compositions of $n$ and the power set (set of all subsets) of $[n-1]$.
The runs in $I(\mu)$ correspond to the components of $\mu$ which are
strictly larger than $1$. 
Each such component $\mu_i$ corresponds to a run of length $\mu_i - 1$.

\medskip
\subsection{Permutations and standard Young tableaux}\label{subsection:perms_SYT}\ 
\medskip

The concepts introduced in this subsection will be used only in
Sections~\ref{section:char_fine} and \ref{section:applications}--\ref{section:fine_sets}.

Let $S_n$ be the symmetric group on the letters $1, \ldots, n$.
It is generated by the simple reflections (adjacent transpositions)
$s_i := (i,i+1)$ $(1 \le i \le n-1)$.

\begin{defn}\label{d.Des_Sn}
The {\em length} $\ell(\pi)$ of a permutation $\pi \in S_n$ is
the minimal number $\ell \ge 0$ such that $\pi$ can be expressed as
a product of $\ell$ generators: $\pi = s_{i_1} \cdots s_{i_\ell}$.
The {\em descent set} of $\pi$ is
\[
\Des(\pi)
:= \{i : \ \pi(i)>\pi(i+1)\}
= \{i : \ell(\pi s_i) < \ell(\pi)\}.
\]
\end{defn}

For a partition $\la$ of $n$ let $\SYT(\la)$ be the set of all standard Young tableaux
of shape $\la$ (see. e.g, \cite[\S 2.5]{Sagan}).

\begin{defn}\label{d.Des_SYT}
The {\em descent set} of a standard Young tableaux $T$ is the set
\[
\Des(T):=\{1 \le i \le n-1 :  i+1 \text{ appears in a lower row of } T \text{ than } i\}.
\]
\end{defn}


%


The {\em Robinson-Schensted correspondence} is a fundamental bijection
mapping each permutation $\pi \in S_n$ to
a pair $(P_\pi, Q_\pi)$ of standard Young tableaux of the same shape
$\la \vdash n$ (which is also called the shape of $\pi$);
for a detailed description see, e.g., \cite[\S 3.1]{Sagan}\cite[\S 7.11]{Stanley_ECII}.

\begin{fact}\label{t.Des_RS}\cite[Fact A3.4.1]{BB}
For every permutation $\pi\in S_n$,
\[
\Des(P_\pi)=\Des(\pi^{-1}) \quad \text{ and } \quad
\Des(Q_\pi)=\Des(\pi).
\]
\end{fact}

\begin{defn}\label{d.Knuth_class}
The {\em Knuth class} corresponding to a standard Young tableau $T$
of size $n$ is the set of permutations
\[
\CCC_T :=\{ \pi \in S_n :\ P_\pi = T \}.
\]
If $T$ has shape $\la$ then
$\CCC_T$ is a {\em Knuth class of shape $\la$}.
\end{defn}

\begin{defn}
The {\em inverse descent class} corresponding to a subset $J\subseteq [n-1]$
is the set of permutations
\[
\{\pi\in S_n:\ \Des(\pi^{-1})=J\}.
\]
\end{defn}

Fact~\ref{t.Des_RS} clearly implies
\begin{corollary}\label{t.inverse_des_class}
An inverse descent class is a disjoint union of Knuth classes.
\end{corollary}


\begin{defn}
A permutation $\pi \in S_n$ is {$k\cdots 21$-avoiding} if the
sequence of values $(\pi(1), \ldots, \pi(n))$ has no decreasing
subsequence of length $k$, i.e., there do not exist indices $1 \le
i_1 < i_2 < \cdots < i_k \le n$ such that $\pi(i_1) > \pi(i_2) >
\cdots > \pi(i_k)$.
\end{defn}


The following classical theorem is due to Schensted.

\begin{theorem}~\cite{Schensted}
For every permutation $\pi \in S_n$, the length of the longest
decreasing subsequence in $\pi$ is equal to the number of rows in
$P_\pi$.
\end{theorem}

\begin{corollary}\label{t.321_avoiding}
For every positive integer $k$, the set of all $k\cdots
21$-avoiding permutations in $S_n$ is a disjoint union of Knuth
classes.
\end{corollary}

\section{Character formulas and fine sets}\label{section:char_fine}

We start by introducing the notion of $\mu$-unimodality
and collecting a number of similarly-looking character formulas,
in order to motivate the forthcoming concept of {\em fine set}.

\medskip
\subsection{$\mu$-unimodality of sets, permutations and tableaux}\label{subsection:prelim_unimodality}\ 
\medskip

%
%
%

Recall from Subsection~\ref{subsection:prelim_compositions}
the definition of $I(\mu)$ for a composition $\mu$,
and the concept of runs in a set of integers.

\begin{defn}\label{def.unimodal_set}
Let $\mu=(\mu_1, \ldots, \mu_t)$ be a composition of $n$.
A subset $J \subseteq [n-1]$ is  {\em $\mu$-unimodal} if
each run of $J \cap I(\mu)$ is a prefix of the corresponding run of $I(\mu)$;
in other words, if
$J \cap I(\mu)$ is a disjoint union of intervals of the form
$\left[ \sum_{i=1}^{k-1}\mu_i+1, \sum_{i=1}^{k-1}\mu_i + \ell_k \right]$,
where $0 \le \ell_k \le \mu_k - 1$ for every  $1 \le k \le t$.
\end{defn}


\begin{defn}\label{def.unimodal_permutation}
Let $\mu=(\mu_1, \ldots, \mu_t)$ be a composition of $n$.
A permutation $\pi \in S_n$ is called {\em $\mu$-unimodal}
if its descent set (see Definition~\ref{d.Des_Sn})
is $\mu$-unimodal in the sense of Definition~\ref{def.unimodal_set}.
Denote by $U_\mu$ the set of all  $\mu$-unimodal permutations in $S_n$.
Thus $\pi \in U_\mu$ if and only if,
for any $1 \le k \le t$, the sequence of values $(\pi(m_k + 1), \ldots, \pi(m_k + \mu_k))$
is unimodal, namely there exists $0 \le \ell_k \le \mu_k - 1$ such that
\[
\pi(m_k + 1) > \ldots > \pi(m_k + \ell_k) > \pi(m_k + \ell_k + 1) < \ldots < \pi(m_k + \mu_k),
\]
where $m_k = \mu_1 + \ldots + \mu_{k-1}$.
This definiton differs slightly from the commonly used notion of unimodality
of sequences,
where all inequalities are reversed; and it includes the extreme cases
$\ell_k = 0$ (an increasing sequence) and $\ell_k = \mu_k - 1$ (a decreasing sequence).
\end{defn}

\begin{example}
The permutation $\pi=936871254$ has $\Des(\pi) = \{1,4,5,8\}$,
and is therefore $(4,3,2)$-unimodal but not $(5,4)$-unimodal.
\end{example}

\begin{defn}\label{def.unimodal_SYT}
Let $\mu=(\mu_1, \ldots, \mu_t)$ be a composition of $n$.
A standard Young tableau $T$ of size $n$ is called {\em $\mu$-unimodal}
if its descent set (see Definition~\ref{d.Des_SYT})
is $\mu$-unimodal in the sense of Definition~\ref{def.unimodal_set}.
For a partition $\la$ of $n$ denote by $U_\mu(\la)$
the set of all  $\mu$-unimodal standard Young tableaux of shape $\la$.
\end{defn}

\begin{example}
The standard Young tableau
\[
T = \,
\young(134,259,68,7)
\]
has $\Des(T) = \{1,4,5,6\}$, and is therefore $(4,5)$-unimodal but not $(5,4)$-unimodal.
\end{example}

\medskip
\subsection{A family of character formulas}\label{subsection:character formulas}\ 
\medskip


Let $\lambda$ be a partition of $n$ and $\mu$ a composition of $n$.
Let $\chi^\lambda$ be the
irreducible $S_n$-character corresponding to $\la$ and,
for any $S_n$-character $\chi$, let $\chi_\mu$ be its value on a conjugacy
class of cycle type $\mu$.
The following formula for the values of irreducible characters is a special case
of~\cite[Theorem 4]{Ro2}; for a direct combinatorial proof see~\cite{Ra2}.
Note that a typo in the formulation of~\cite[Theorem 4]{Ro2} was corrected in~\cite{Ra2}.
Note also that the original statements of all the formulas quoted in this subsection
formally assumed that $\mu$ is a partition, but their proofs are actually valid
whenever $\mu$ is an arbitrary composition.

\begin{theorem}\label{t.c1}%
\cite[Theorem 4]{Ro2}\cite[Equations (0.2)--(0.4)]{Ra2}
\[
\chi^\lambda_\mu = \sum_{T\in U_\mu(\lambda)} (-1)^{|\Des(T) \cap I(\mu)|},
\]
where $U_\mu(\lambda)$ is the set of all $\mu$-unimodal standard
Young tableaux of shape $\lambda$.
\end{theorem}

\medskip

Let $\chi^{(k)}$ be the $S_n$-character defined by the symmetric group action on
the $k$-th homogeneous component of the coinvariant algebra. Then

\begin{theorem}\label{t.c2}\cite[Theorem 1]{Ro-Schubert}
\[
\chi^{(k)}_\mu =
\sum_{\pi\in L(k)\cap U_\mu} (-1)^{|\Des(\pi) \cap I(\mu)|},
\]
where $L(k)$ is the set of all permutations of length $k$ in $S_n$.
\end{theorem}

\medskip

\begin{corollary}\label{t.c_regular}
Let $\chi^{S_n}$ be the character of the regular representation of
$S_n$. Then
\[
\chi^{S_n}_\mu = \sum_{\pi\in U_\mu} (-1)^{|\Des(\pi) \cap I(\mu)|},
\]
where the sum runs over all $\mu$-unimodal permutations in $S_n$.
\end{corollary}

\begin{proof}
By Theorem~\ref{t.c2}, the RHS is the $S_n$-character of the coinvariant
algebra of $S_n$, which is isomorphic as an $S_n$-module to the
group algebra $\CC[S_n]$~\cite{Chevalley}. 
\end{proof}

\medskip

A complex representation of a group or an algebra $A$ is called a
{\it Gelfand model} for $A$ if it is equivalent to the
multiplicity-free direct sum of all the irreducible $A$-representations.
Let $\chi^G$ be the character of the Gelfand model of $S_n$ (or of its group algebra).

\begin{theorem}\label{t.c3}\cite[Proposition 1.5]{APR2}
\[
\chi^G_\mu = \sum_{\pi\in I_n\cap U_\mu} (-1)^{|\Des(\pi) \cap I(\mu)|},
\]
where $I_n:=\{\pi \in S_n : \pi^2 = id\}$ is the set of all involutions in $S_n$.
\end{theorem}

Note that there is a misprint in the formulation of~\cite[Proposition 1.5]{APR2},
which is corrected here.



\medskip

In this paper we propose a general setting for all of these results.
In particular, we provide an answer to the following question.

\begin{question}\label{q.invertible}
Are these character formulas invertible?
\end{question}

\medskip
\subsection{Fine sets}\label{subsection:concept}\ 
\medskip

\begin{defn}\label{d.fine}
Let $\BBB$ be a set of combinatorial objects, and let $\Des: \BBB \to P_{n-1}$ be
a map which associates with each element $b\in \BBB$ a subset $\Des(b) \subseteq [n-1]$.
Denote by $\BBB_\mu$ the set of elements in $\BBB$ whose ``descent set'' $\Des(b)$ is
$\mu$-unimodal, as in Definition~\ref{def.unimodal_set}.
Let $\rho$ be a complex $S_n$-representation, with a corresponding character $\chi^\rho$.
Then $\BBB$ is called a {\em fine set} for $\rho$ (with respect to the map $\Des$)
if, for each composition $\mu$ of $n$, the value of the character $\chi^\rho$
at a conjugacy class of cycle type $\mu$ satisfies
\begin{equation}\label{defn-fine1}
\chi^\rho_\mu = \sum_{b \in \BBB_\mu} (-1)^{|\Des(b) \cap I(\mu)|} .
\end{equation}
\end{defn}

If the map $\Des$ is understood from the context (as in the case $\BBB\subseteq  S_n$),
we simply say that $\BBB$ is a fine set for $\rho$.
Furthermore, $\BBB$ is called a fine set if it is a fine set for some representation of $S_n$.

\medskip

The following proposition restates Theorems~\ref{t.c1}, \ref{t.c2} and~\ref{t.c3}
in terms of fine sets (with respect to the ordinary descent maps on
permutations and on standard Young tableaux, as in
Definitions~\ref{d.Des_Sn} and~\ref{d.Des_SYT}).

\begin{proposition}\label{t.fine-examples}\
\begin{itemize}
\item[$(i)$]
For any partition $\la$ of $n$,
the set of standard Young tableaux of shape 
$\la$ is a fine set for the 
irreducible module $S^\la$.
\item[$(ii)$]
The set of permutations of fixed length $k$ in $S_n$ is a fine set for
the $k$-th homogeneous component of the coinvariant algebra of $S_n$.
\item[$(iii)$]
The set of involutions in $S_n$ is a fine set for the Gelfand model of $S_n$.
\end{itemize}
\end{proposition}

More examples of fine sets will be given in Section~\ref{section:fine_sets}.

\section{Two families of matrices}\label{section:matrices}

We now introduce a linear algebraic construction which will help us
answer Question~\ref{q.invertible}.

It is well known that {\em partitions} of $n$ are the natural indices for the
characters and conjugacy classes of $S_n$.
It turns out that a major step towards an answer to Question~\ref{q.invertible} is
to use, instead, {\em compositions} of $n$ -- or, equivalently, subsets of $[n-1]$ --
in spite of the apparent redundancy.
A surprising structure arises, in the form of a certain matrix $A_{n-1}$.
In fact, it is convenient to 
define two ``coupled'' families of matrices, $(A_n)$ and $(B_n)$;
for each nonnegative integer $n$, $A_n$ and $B_n$ are square
matrices of order $2^n$,  with entries $0$, $\pm 1$,
which may be viewed as asymmetric variants of Walsh-Hadamard matrices.
These matrices and some of their properties will be presented in
the following sections.

We shall give two equivalent definitions for these families of matrices.
The explicit definition is closer in spirit to the subsequent applications,
but the recursive definition is very simple to describe and easy to use,
and will therefore be presented first.

\medskip
\subsection{A recursive definition}\ 
\medskip

Recall the well known {\em Walsh-Hadamard (Sylvester)} matrices, defined by the recursion
\[
H_n = \left(\begin{array}{cc}
H_{n-1} & H_{n-1} \\
H_{n-1} & -H_{n-1}
\end{array}\right)
\qquad(n \ge 1)
\]
with $H_0 = (1)$.

\begin{defn}\label{d.AB_recursion}
Define, recursively,
\[
A_n = \left(\begin{array}{cc}
A_{n-1} & A_{n-1} \\
A_{n-1} & -B_{n-1}
\end{array}\right)
\qquad(n \ge 1)
\]
with $A_0 = (1)$, and
\[
B_n = \left(\begin{array}{cc}
A_{n-1} & A_{n-1} \\
0 & -B_{n-1}
\end{array}\right)
\qquad(n \ge 1)
\]
with $B_0 = (1)$.
\end{defn}

Each of the matrices $A_n$ and $B_n$ may be obtained from the corresponding
Walsh-Hadamard matrix $H_n$, all the entries of which are $\pm 1$,
by replacing some of the entries by $0$.

\begin{example}
\[
A_1 = \left(\begin{array}{cc}
1 & 1 \\
1 & -1
\end{array}\right)
\qquad B_1 = \left(\begin{array}{cc}
1 & 1 \\
0 & -1
\end{array}\right)
\]

\[
A_2 = \left(\begin{array}{cccc}
1 & 1  & 1 & 1 \\
1 & -1  & 1 & -1\\
1 & 1 & -1 & -1 \\
1 & -1 & 0 & 1
\end{array}\right)
\qquad B_2 = \left(\begin{array}{cccc}
1 & 1  & 1 & 1 \\
1 & -1  & 1 & -1 \\
0 & 0 & -1 & -1 \\
0 & 0 & 0 & 1
\end{array}\right)
\]
\end{example}


\medskip
\subsection{An explicit definition}\ 
\medskip


It will be convenient to index the rows and columns of $A_n$ and
$B_n$ by subsets of the set $[n] = \{1,\ldots,n\}$.

\begin{defn}\label{d.ord}
For any finite set $J$ of positive integers, let $\ord(J) :=\sum_{j \in J} 2^{j-1}$.
\end{defn}

\begin{defn}\label{def.Pn}
For any nonnegative integer $n$ let $P_n$ be the power set (set of all subsets) of $[n]$.
The parameter $\ord$ maps $P_n$ bijectively to $\{0, 1, \ldots, 2^n - 1\}$,
thus defining a linear order on $P_n$ known as the {\em anti-lexicographic} order:
For $I, J \in P_n$, $I \ne J$, if $m$ is the largest element in
the symmetric difference $I \triangle J := (I \cup J) \setminus (I \cap J)$,
then $I < J \iff m \in J$.
\end{defn}

\begin{example}
The linear order on $P_3$ is
\[
\emptyset < \{1\} < \{2\} < \{1, 2\} < \{3\} < \{1, 3\} < \{2, 3\} < \{1, 2, 3\}.
\]
\end{example}

Order $P_n$ as in Definition~\ref{def.Pn}.
The entries of the Walsh-Hadamard matrix $H_n = (h_{I,J})_{I, J \in P_n}$
are explicitly given by the formula
\[
h_{I,J} := (-1)^{|I \cap J|} \qquad(\forall I, J \in P_n).
\]

Recall, from Subsection~\ref{subsection:prelim_compositions}, the definition of a prefix of an interval.

\begin{lemma}\label{t.AB_explicit}{\rm (Explicit definition)}
Order $P_n$ as in Definition~\ref{def.Pn},
and let $I_1, \ldots, I_t$ be the runs of $I\in P_n$
(see Definition~\ref{def.runs}).
Then: 
\begin{itemize}
\item[$(i)$]
$A_n = (a_{I,J})_{I, J \in P_n}$, where
\[
a_{I,J} =
\begin{cases}
(-1)^{|I \cap J|}, & \text{if $I_k \cap J$ is a prefix of $I_k$ for each $k$;}\\
0, & \text{otherwise.}
\end{cases}
\]
\item[$(ii)$]
$B_n = (b_{I,J})_{I, J \in P_n}$, where:
\[
b_{I,J} =
\begin{cases}
(-1)^{|I \cap J|}, & \text{if $I_k \cap J$ is a prefix of $I_k$ for each $k$, and}\\
& n \not\in I \setminus J;\\
0, & \text{otherwise.}
\end{cases}
\]
\end{itemize}

\end{lemma}

\begin{proof}
It will be convenient here to define $A_n$ and $B_n$ explicitly as in the lemma,
and then show that they satisfy the recursions in Definition~\ref{d.AB_recursion}.

We shall start with $A_n$. Clearly $A_0 = (1)$.

For $I, J \in P_n$ ($n \ge 1$)
denote $I' := I \setminus \{n\}$ and $J' := J \setminus \{n\}$.

The ``upper left'' quarter of $A_n$ corresponds to $I, J \in P_n$
such that $n \not\in I$ and $n \not\in J$. In this case, clearly
$a_{I,J}$ in $A_n$ is the same as $a_{I',J'}$ in $A_{n-1}$.

Similarly when $n \not\in I$ and $n \in J$, and also when $n \in I$ and $n \not\in J$:
$|I \cap J| = |I' \cap J'|$, and $I_k \cap J$ is a prefix of $I_k$ for all $k$
if and only if $I'_k \cap J$ is a prefix of $I'_k$ for all $k$.

The ``lower right'' quarter of $A_n$ corresponds to $I, J \in P_n$
such that $n \in I \cap J$.
If $n-1 \not\in I$ then $I_k \cap J$ is a prefix of $I_k$ for all $k$
if and only if $I'_k \cap J$ is a prefix of $I'_k$ for all $k$.
Also $|I \cap J| = |I' \cap J'| + 1$, so that $a_{I,J}$ in $A_n$ is equal to
$-a_{I',J'}$ in $A_{n-1}$ and also to $-b_{I',J'}$ in $B_{n-1}$
(since $n-1 \not\in I'$ so $n-1 \not\in I' \setminus J'$).
If $n-1 \in I \cap J$ then, again, $a_{I,J}$ in $A_n$ is equal to
$-a_{I',J'}$ in $A_{n-1}$ and also to $-b_{I',J'}$ in $B_{n-1}$
(since $n-1 \in J'$ so $n-1 \not\in I' \setminus J'$).
Finally, if $n-1 \in I$ but $n-1 \not\in J$ then, for the last run $I_t$ of $I$,
$I_t \cap J$ is not a prefix of $I_t$, and thus $a_{I,J} = 0$ in $A_n$ as well as
$-b_{I',J'} = 0$ in $B_{n-1}$ (since $n-1 \in I' \setminus J'$).

We have proved the recursion for $A_n$.
The entries of $B_n$ are equal to the corresponding entries of $A_n$,
except for those in the quarter corresponding to $(I,J)$ with $n \in I$ and $n \not\in J$,
which are all zeros  (since $n \in I \setminus J$). This proves the recursion for $B_n$ as well.

\end{proof}

\medskip
\subsection{Determinants}\ 
\medskip

It turns out that the invertibility of $A_n$ is the key factor in an answer
to Question~\ref{q.invertible} (see Observation~\ref{t.fine_matrix}).
We can prove invertibility by actually computing the determinant.

\begin{theorem}\label{t.An-determinant}
$A_n$ and $B_n$ are invertible for all $n \ge 0$. In fact,
\[
\det(A_n) = (n+1) \cdot \prod_{k=1}^{n} k^{2^{n-1-k} (n+4-k)} \qquad(n \ge 2)
\]
while $\det(A_0) = 1$ and $\det(A_1) = -2$, and
\[
\det(B_n) = \prod_{k=1}^{n} k^{2^{n-1-k} (n+2-k)} \qquad(n \ge 2)
\]
while $\det(B_0) = 1$ and $\det(B_1) = -1$.
\end{theorem}

A proof of  Theorem~\ref{t.An-determinant} will be given in
Section~\ref{section:AM_BM}. For comparison,
\[
\det(H_n) = 2^{2^{n - 1} n} \qquad(n \ge 2)
\]
with $\det(H_0) = 1$ and $\det(H_1) = -2$.

\medskip
\subsection{Eigenvalues}\ 
\medskip

Consider the matrix
\[
A_2 = \left(\begin{array}{cccc}
1 & 1  & 1 & 1 \\
1 & -1  & 1 & -1\\
1 & 1 & -1 & -1 \\
1 & -1 & 0 & 1
\end{array}\right).
\]
As an asymmetric matrix, it might conceivably have non-real eigenvalues.
Surprisingly, computation shows that its characteristic polynomial is
\[
(x^2 - 3)(x^2 - 4),
\]
and thus all its eigenvalues are (up to sign) square roots of
positive integers!



This is not a coincidence.
The following combinatorial description of the eigenvalues of $A_n$ and $B_n$,
which was stated as a conjecture in an earlier version of this paper,
has recently been proved by Gil Alon.

\begin{theorem}\label{t.GilAlon} {\rm (G.\ Alon~\cite{Alon})}\ 
\begin{itemize}
\item[$(i)$]
The roots of the characteristic polynomial of $A_n$ are in $2:1$ correspondence
with the compositions of $n$: each composition $\mu = (\mu_1, \ldots, \mu_t)$ of $n$
corresponds to a pair of eigenvalues $\pm \sqrt{\pi_\mu}$ of $A_n$, where
\[
\pi_\mu := \prod_{i=1}^{t} (\mu_i+1).
\]
\item[$(ii)$]
The roots of the characteristic polynomial of $B_n$ are in $2:1$ correspondence
with the compositions of $n$: each composition $\mu = (\mu_1, \ldots, \mu_t)$ of $n$
corresponds to a pair of eigenvalues $\pm \sqrt{\pi'_\mu}$ of $B_n$, where
\[
\pi'_\mu := \prod_{i=1}^{t-1} (\mu_i+1).
\]
\end{itemize}
\end{theorem}

Surprising connections between the eigenvalues and the diagonal elements of $A_n^2$,
as well as the column sums of some related matrices, appear in Theorem~\ref{t.col_sums}
and Corollary~\ref{t.ev_eq_diag} below.


\section{M\"obius inversion}\label{section:AM_BM}

In this section we prove Theorem~\ref{t.An-determinant}.
Our approach is to study certain matrices, with more transparent structure, obtained from
$A_n$ and $B_n$ by a transformation corresponding to (poset theoretic) M\"obius inversion.

\medskip
\subsection{Auxiliary definitions}\ 
\medskip

Let us define certain auxiliary families of matrices.

\begin{defn}\label{d.ZM_recursion}
Define, recursively,
\[
Z_n = \left(\begin{array}{cc}
Z_{n-1} & Z_{n-1} \\
0 & Z_{n-1}
\end{array}\right)
\qquad(n \ge 1)
\]
with $Z_0 = (1)$, as well as
\[
M_n = \left(\begin{array}{cc}
M_{n-1} & -M_{n-1} \\
0 & M_{n-1}
\end{array}\right)
\qquad(n \ge 1)
\]
with $M_0 = (1)$.
\end{defn}

$Z_n$ is the {\em zeta matrix} of the poset $P_n$ with respect
to {\em set inclusion} (not with respect to its linear extension,
described in Definition~\ref{def.Pn}).
Thus $Z_n = (z_{I,J})_{I,J \in P_n}$ is a square matrix, with entries satisfying
\[
z_{I,J} =
\begin{cases}
1, & \text{if $I \subseteq J$;}\\
0, & \text{otherwise.}
\end{cases}
\]

$M_n = Z_n^{-1}$ is the corresponding {\em M\"obius matrix},
expressing the M\"obius function (see~\cite{Rota}) of the poset $P_n$.
Thus $M_n = (m_{I,J})_{I,J \in P_n}$ has entries satisfying
\[
m_{I,J} =
\begin{cases}
(-1)^{|J \setminus I|}, & \text{if $I \subseteq J$;}\\
0, & \text{otherwise.}
\end{cases}
\]

\begin{defn}
Denote $AM_n := A_n M_n$, $BM_n := B_n M_n$ and $HM_n := H_n M_n$.
\end{defn}

It follows from Definitions~\ref{d.AB_recursion} and~\ref{d.ZM_recursion} that
\begin{equation}\label{eq.AMn}
AM_n = \left(\begin{array}{cc}
AM_{n-1} & 0 \\
AM_{n-1} & -(AM_{n-1} + BM_{n-1})
\end{array}\right)
\qquad(n \ge 1)
\end{equation}
with $AM_0 = (1)$ and
\begin{equation}\label{eq.BMn}
BM_n = \left(\begin{array}{cc}
AM_{n-1} & 0 \\
0 &  -BM_{n-1}
\end{array}\right)
\qquad(n \ge 1)
\end{equation}
with $BM_0 = (1)$, as well as 
\begin{equation}\label{eq.HMn}
HM_n = \left(\begin{array}{cc}
HM_{n-1} & 0 \\
HM_{n-1} & -2HM_{n-1}
\end{array}\right)
\qquad(n \ge 1)
\end{equation}
with $HM_0 = (1)$.

The block triangular form of $AM_n$ and block diagonal form of $BM_n$ facilitate
a recursive computation of the determinants of $A_n$ and $B_n$.

\medskip
\subsection{A proof of Theorem~\ref{t.An-determinant}}\ 
\medskip

By recursion~(\ref{eq.BMn}),
\[
\det(BM_n) = \det(AM_{n-1}) \det(-BM_{n-1}) \qquad(n \ge 1).
\]
Now $M_n$ is an upper triangular matrix with $1$-s on its diagonal, so that
\[
\det(M_n) = 1 \qquad (n \ge 0).
\]
We conclude that
\begin{equation}\label{eq.Bn}
\det(B_n) = \det(A_{n-1}) \det(B_{n-1}) \qquad(n\ge 2)
\end{equation}
while
\[
\det(B_1) = - \det(A_0) \det(B_0).
\]
Similarly, for any scalar $t$,
\[
AM_n + tBM_n = \left(\begin{array}{cc}
(t+1)AM_{n-1} & 0 \\ AM_{n-1} & -AM_{n-1} - (t+1)BM_{n-1}
\end{array}\right)
\qquad(n \ge 1)
\]
and a similar argument yields
\[
\det(A_n + tB_n) = \det((t+1)A_{n-1})
\det(A_{n-1} + (t+1)B_{n-1}) \qquad(n\ge 2)
\]
and
\[
\det(A_1 + tB_1) = - \det((t+1)A_0)
\det(A_0 + (t+1)B_0).
\]

It follows that
\begin{align*}
\det(A_n)
&= \det(A_{n-1}) \det(A_{n-1} + B_{n-1}) \\
&=\det(A_{n-1}) \det(2 A_{n-2}) \det(A_{n-2} + 2B_{n-2}) \\
&= \ldots \\
& = - \left( \prod_{k=1}^{n} \det(k A_{n-k}) \right) \cdot \det(A_0 + n B_0) = \\
& = - (n+1) \cdot \prod_{k=1}^{n} k^{2^{n-k}} \cdot \prod_{k=1}^{n} \det(A_{n-k}) \qquad(n \ge 1).
\end{align*}
Since $A_0 = (1)$ it follows that $\det(A_n) \ne 0$ for any
nonnegative integer $n$, and therefore
\[
\det(A_n)  / \det(A_{n-1}) =
\frac{-(n+1)}{-n} \cdot n \cdot \prod_{k=1}^{n-1} k^{2^{n-1-k}} \cdot \det(A_{n-1}) \qquad(n\ge 2).
\]
The solution to this recursion, with initial value $\det(A_1) = -2$, is
\[
\det(A_n) = (n+1) \cdot \prod_{k=1}^{n} k^{2^{n-1-k} (n+4-k)} \qquad(n \ge 2).
\]
Recursion (\ref{eq.Bn}) above, with initial value $\det(B_1) = -1$, now yields
\[
\det(B_n) = \prod_{k=1}^{n} k^{2^{n-1-k} (n+2-k)} \qquad(n \ge 2).
\]
For comparison, by recursion~\ref{eq.HMn},
\[
\det(H_n) = 2^{2^{n-1}} \det(H_{n-1})^2 \qquad(n\ge 2)
\]
with initial value $\det(H_1) = -2$, so that
\[
\det(H_n) = 2^{2^{n - 1} n} \qquad(n \ge 2).
\]
\qed

\begin{remark}
We can also write
\[
\det(A_n) = \prod_{k=1}^{n+1} k^{a_{n+1-k}}\qquad(n \ge 2),
\]
where the sequence $(a_0, a_1, \ldots) = (1, 2, 5, 12, 28, 64, \ldots)$
coincides with \cite[sequence A045623]{Sloane}.
\end{remark}

\section{Properties of the transformed matrices}\label{section:matrix_entries}

In this section we describe some additional properties of
the transformed matrices $AM_n$ and $BM_n$, namely:
Explicit expressions for their entries,
for the entries of their inverses,
and for their row sums and column sums.
The latter are surprisingly related to the eigenvalues of $A_n$ and $B_n$
described in Theorem~\ref{t.GilAlon} above.

The proofs of most of the results in this section follow from
recursion formulas (\ref{eq.AMn}) and (\ref{eq.BMn}), 
and are therefore indicated only when additional ingredients are present.

\medskip
\subsection{Matrix entries}\ 
\medskip

We shall now compute explicitly the entries of $AM_n$
and $BM_n$, starting with $HM_n$ as a ``baby case''.

\begin{example}
\[
HM_3 = \left(\begin{array}{cccccccc}
1 & 0 & 0 & 0 & 0 & 0 & 0 & 0 \\
1 & -2 & 0 & 0 & 0 & 0 & 0 & 0 \\
1 & 0 & -2 & 0 & 0 & 0 & 0 & 0 \\
1 & -2 & -2 & 4 & 0 & 0 & 0 & 0 \\
1 & 0 & 0 & 0 & -2 & 0 & 0 & 0 \\
1 & -2 & 0 & 0 & -2 & 4 & 0 & 0 \\
1 & 0 & -2 & 0 & -2 & 0 & 4 & 0 \\
1 & -2 & -2 & 4 & -2 & 4 & 4 & -8
\end{array}\right)
\]
\end{example}
This generalizes to an explicit description of the entries of $HM_n$,
which follows easily from recursion~(\ref{eq.HMn}).


\begin{lemma}\label{t.HMentries}\
\[
(HM_n)_{I,J} \ne 0 \iff J \subseteq I
\]
and
\[
(HM_n)_{I,J} \ne 0 \,\then\, (HM_n)_{I,J} = (-2)^{|J|}.
\]
\end{lemma}

The corresponding results for $AM_n$ and $BM_n$ are much more subtle
(and interesting). Their proofs follow, in general, from recursions (\ref{eq.AMn}) and (\ref{eq.BMn}).

\begin{lemma}\label{t.LT}
For every $n\ge 0$, the matrices $AM_n$ and $BM_n$ are lower
triangular.
\end{lemma}


\begin{corollary}
$(AM_n) \cdot Z_n$ is an $LU$ factorization of $A_n$; similarly
for $B_n$.
\end{corollary}

\begin{example}
\[
AM_1 = \left(\begin{array}{cc}
1 & 0 \\
1 & -2
\end{array}\right)
\qquad BM_1 = \left(\begin{array}{cc}
1 & 0 \\
0 & -1
\end{array}\right)
\]

\[
AM_2 = \left(\begin{array}{cccc}
1 & 0 & 0 & 0 \\
1 & -2 & 0 & 0 \\
1 & 0 & -2 & 0 \\
1 & -2 & -1 & 3
\end{array}\right)
\qquad BM_2 = \left(\begin{array}{cccc}
1 & 0 & 0 & 0 \\
1 & -2 & 0 & 0 \\
0 & 0 & -1 & 0 \\
0 & 0 & 0 & 1
\end{array}\right)
\]

\[
AM_3 = \left(\begin{array}{cccccccc}
1 & 0 & 0 & 0 & 0 & 0 & 0 & 0 \\
1 & -2 & 0 & 0 & 0 & 0 & 0 & 0 \\
1 & 0 & -2 & 0 & 0 & 0 & 0 & 0 \\
1 & -2 & -1 & 3 & 0 & 0 & 0 & 0 \\
1 & 0 & 0 & 0 & -2 & 0 & 0 & 0 \\
1 & -2 & 0 & 0 & -2 & 4 & 0 & 0 \\
1 & 0 & -2 & 0 & -1 & 0 & 3 & 0 \\
1 & -2 & -1 & 3 & -1 & 2 & 1 & -4
\end{array}\right)
\]
\[
BM_3 = \left(\begin{array}{cccccccc}
1 & 0 & 0 & 0 & 0 & 0 & 0 & 0 \\
1 & -2 & 0 & 0 & 0 & 0 & 0 & 0 \\
1 & 0 & -2 & 0 & 0 & 0 & 0 & 0 \\
1 & -2 & -1 & 3 & 0 & 0 & 0 & 0 \\
0 & 0 & 0 & 0 & -1 & 0 & 0 & 0 \\
0 & 0 & 0 & 0 & -1 & 2 & 0 & 0 \\
0 & 0 & 0 & 0 & 0 & 0 & 1 & 0 \\
0 & 0 & 0 & 0 & 0 & 0 & 0 & -1
\end{array}\right)
\]
\end{example}

Apparently, $AM_n$ has the same zero pattern and the same sign
pattern as $HM_n$. $BM_n$ also has the same sign pattern, but has
more zero entries. The absolute values of entries in both families
are more intricate than in $HM_n$.

\begin{theorem}\label{t.AB_entries} {\rm (Entries of $AM_n$ and $BM_n$)}
\begin{itemize}
\item[$(i)$] {\bf Zero pattern:}
\[
(AM_n)_{I,J} \ne 0 \iff J \subseteq I
\]
and
\[
(BM_n)_{I,J} \ne 0 \iff J \subseteq I \text{\ and\ } \fin(J) =
\fin(I),
\]
where
\[
\fin(I) := \max\{0 \le i \le n \,|\, i \not\in I\} \qquad(\forall
I \in P_n).
\]
\item[$(ii)$] {\bf Signs:}
\[
(AM_n)_{I,J} \ne 0 \then \sign ((AM_n)_{I,J}) = (-1)^{|J|}
\]
and
\[
(BM_n)_{I,J} \ne 0 \then \sign ((BM_n)_{I,J}) = (-1)^{|J|}.
\]
\item[$(iii)$] {\bf Absolute values:}
For $I, J \in P_n$, let $J_1,
\ldots, J_t$ be the runs (maximal consecutive intervals) in $J$.
For $J_k = \{m_k + 1, \ldots, m_k + \ell_k\}$ $(1 \le k \le t)$,
let
\[
c_k(I) =
\begin{cases}
0, & \text{if $m_k \in I$;} \\
1, & \text{otherwise.}
\end{cases}
\]
Then
\[
(AM_n)_{I,J} \ne 0 \then |(AM_n)_{I,J}| = \prod_{k=1}^{t} (|J_k| +
1)^{c_k(I)}
\]
and
\[
(BM_n)_{I,J} \ne 0 \then |(BM_n)_{I,J}| = \prod_{k=1}^{t'} (|J_k|
+ 1)^{c_k(I)},
\]
where
\[
t' =
\begin{cases}
t-1, & \text{if $n \in I$ (equivalently, $n \in J$);} \\
t, & \text{otherwise.}
\end{cases}
\]
\end{itemize}
\end{theorem}

\begin{proof}
It is clear from recursion formulas~(\ref{eq.AMn})
and~(\ref{eq.BMn}) that all the entries in column $J$ of $AM_n$
(or $BM_n$) have sign $(-1)^{|J|}$ or are zero, exactly as in
$HM_n$.

Comparison of the two recursions shows that wherever $AM_n$ has a
zero entry so does $BM_n$, but not conversely. The zero pattern of
$AM_n + BM_n$ is therefore the same as that of $AM_n$, and thus
recursions (\ref{eq.AMn}) and~(\ref{eq.HMn}) imply that $AM_n$ and
$HM_n$ have the same zero pattern. The zero pattern of $BM_n$ now
follows from recursion~(\ref{eq.BMn}).

Finally, the explicit formulas for the absolute values of entries
are relevant, of course, only when $J \subseteq I$.
They are a little difficult to come up with, but easy to confirm by recursion.
\end{proof}

\begin{corollary}
Let $I, J \in P_n$ satisfy $J \subseteq I$. Then:
\begin{itemize}
\item[$(i)$]
\[
|(AM_n)_{I,J}| \le  |(HM_n)_{I,J}| = 2^{|J|},
\]
with equality if and only if $|J_k| = 1$ for each $k$ for which
$m_k \not\in I$.
\item[$(ii)$]
\[
|(BM_n)_{I,J}| \le  |(AM_n)_{I,J}|,
\]
with equality if and only if either $n \not\in I$ or $m_t \in I$.
\end{itemize}
\end{corollary}

\begin{proof}
$|J_k| + 1 \le 2^{|J_k|}$, with equality if and only if $|J_k| =
1$.
\end{proof}

An alternative description of the entries may be given in terms of compositions,
using the correspondence $\mu \longleftrightarrow I(\mu)$ described in
Subsection~\ref{subsection:prelim_compositions} above.

%
%

\begin{theorem}\label{t.AB_entries_comp} {\rm (Entries of $AM_n$ and $BM_n$, composition version)}\\
Let $\la$ and $\mu$ be compositions of $n+1$.
Write $(AM_n)_{\la, \mu}$ instead of $(AM_n)_{I(\la), I(\mu)}$, and similarly for $BM_n$.
\begin{itemize}
\item[$(i)$] {\bf Zero pattern:}
\[
(AM_n)_{\la,\mu} \ne 0 \iff \text{\rm $\mu$ is a refinement of $\la$}
\]
and
\begin{eqnarray*}
(BM_n)_{\la,\mu} \ne 0 &\iff& \text{\rm $\mu$ is a refinement of $\la$ and}\\
& & \text{\rm the last component of $\la$ is unrefined in $\mu$}.
\end{eqnarray*}
\item[$(ii)$] {\bf Signs:}
\[
(AM_n)_{\la,\mu} \ne 0 \,\then\, \sign ((AM_n)_{\la,\mu}) = (-1)^{n+1-\ell(\mu)}
\]
and
\[
(BM_n)_{\la,\mu} \ne 0 \,\then\, \sign ((BM_n)_{\la,\mu}) = (-1)^{n+1-\ell(\mu)},
\]
where  $\ell(\mu)$ is the number of components of $\mu$.
\item[$(iii)$] {\bf Absolute values:}
\[
(AM_n)_{\la,\mu} \ne 0 \,\then\, |(AM_n)_{\la, \mu}| = \prod_i \mu_{\rm init}(\la_i)
\]
and
\[
(BM_n)_{\la,\mu} \ne 0 \,\then\, |(BM_n)_{\la, \mu}| = {\prod_i }'
\mu_{\rm init}(\la_i),
\]
where $\mu_{\rm init}(\la_i)$ is the first component in the subdivision (in $\mu$)
of the component $\la_i$ of $\la$,
and $\prod_i '$ is a product over all values of $i$ except the last one.
\end{itemize}
\end{theorem}

\medskip
\subsection{Diagonal entries}\ 
\medskip

The following corollary of Theorem~\ref{t.AB_entries}$(iii)$
and Theorem~\ref{t.AB_entries_comp}$(iii)$
is stated, simultaneously, in terms of a composition $\mu$ of $n+1$
and the corresponding subset $J = I(\mu)$ of $[n]$.
Different indices ($i$ and $k$) are used for the components of $\mu$ and the runs of $J$,
since runs correspond only to the components statisfying $\mu_i > 1$.

\begin{corollary}\label{t.AB_diag} {\rm (Diagonal and last row of $AM_n$)}
\begin{itemize}
\item[$(i)$]
The diagonal entries of $AM_n$ are
\[
|(AM_n)_{J,J}| = \prod_i \mu_i = \prod_k (|J_k| + 1)
\]
and the entries in its last row are
\[
|(AM_n)_{[n],J}| = \mu_1 =
\begin{cases}
|J_1| + 1, & \text{if $1 \in J$;}\\
1, & \text{otherwise.}
\end{cases}
\]
\item[$(ii)$]
Each nonzero entry $(AM_n)_{I,J}$ divides
the corresponding diagonal entry $(AM_n)_{J,J}$ and is divisible by
the corresponding last row entry $(AM_n)_{[n],J}$.
\end{itemize}
\end{corollary}

Recall, from Definition~\ref{d.ord}, the bijection $J \longleftrightarrow \ord(J)$
between finite sets of integers $J$ and nonnegative integers.
Define
\begin{equation*}
a_{\ord(J)} := |(AM_n)_{J,J}| \qquad (\forall J \in P_n),
\end{equation*}
noting that this number depends on $J$ but not on $n$,
and consider the resulting sequence $(a_m)_{m \ge 0}$ of absolute
values of diagonal entries of the ``limit matrix''
$AM_{\infty} = \lim_{n \to \infty} AM_n$.

\begin{lemma}
The sequence $(a_m)_{m \ge 0}$ satisfies the recursive definition
\[
a_0 = 0,\ \ \ a_{2m} = a_m,\ \ \ a_{4m+1} = 2a_{2m},\ \ \ a_{4m+3}
= 2a_{2m+1} - a_m \quad(\forall m \ge 0).
\]
\end{lemma}

\begin{proof}
Consider the formula in Corollary~\ref{t.AB_diag}$(i)$ expressing a diagonal entry
of $AM_n$ in terms of the corresponding run lengths.
We shall not distinguish a set $J$ from its ordinal number $\ord(J)$.

(The set corresponding to) $2m$ has the same runs as $m$, shifted forward by $1$,
so that $a_{2m} = a_m$.
$4m+1$ has the same runs as $2m$, shifted forward by $1$,
plus a singleton run $\{1\}$, 
so that $a_{4m+1} = 2a_{2m}$.

If $m$ is even then $a_{4m+3} = 3a_m$ and $a_{2m+1} = 2a_m$,
so that $a_{4m+3} = 2a_{2m+1} - a_m$.
If $m$ is odd, let $\ell$ be the length of the first run in $m$.
The corresponding runs in $2m+1$ and in $4m+3$ have lengths $\ell+1$ and
$\ell+2$, respectively, so that again $a_{4m+3} = 2a_{2m+1} - a_m$.
\end{proof}

\begin{corollary}
The sequence $(a_m)$ coincides with \cite[sequence A106737]{Sloane}.
Thus, in particular,
\[
a_m = \sum_{k=0}^m \left[ {m+k \choose m-k} {m \choose k} \mod 2 \right],
\]
where the expression in the square brackets is interpreted as
either $0$ or $1$ and summed as an ordinary integer.
\end{corollary}


\medskip
\subsection{Row sums and column sums}\ 
\medskip

The following two results, regarding row and column sums of $AM_n$ and $BM_n$,
are stated, for simplicity, almost entirely in the language of compositions.

\begin{lemma}\label{t.row_sums} {\rm (Row sums of $AM_n$, $BM_n$)}\\
Let $\la$ be a composition of $n+1$, and let $I = I(\la)$ be the
corresponding subset  of $[n]$.
\begin{itemize}
\item[$(i)$]
The sum of all entries in row $I$ of $AM_n$ (or $BM_n$, or $HM_n$) is $(-1)^{|I|}$.
\item[$(ii)$]
The sum of absolute values of all entries in row $I$ of $AM_n$ is
\[
\prod_i (2^{\la_i} - 1).
\]
The sum of absolute values of all entries in row $I$ of $BM_n$ is
\[
{\prod_i}' (2^{\la_i} - 1),
\]
where $\prod_i '$ is a product over all values of $i$ except the last.
In $HM_n$ the corresponding sum is $3^{|I|}$.
\end{itemize}
\end{lemma}

\begin{proof}
The sum of all entries in row $I$ of $AM_n$ is
\[
\sum_J (AM_n)_{I,J} = \sum_J (AM_n)_{I,J} (Z_n)_{J,[n]}
= (AM_n \cdot Z_n)_{I,[n]} = (A_n)_{I,[n]} = (-1)^{|I|},
\]
and similarly for $BM_n$ and $HM_n$.

In order to compute the sum of absolute values,
recall that $I = I(\la)$ for a composition $\la = (\la_1, \ldots, \la_t)$ of $n+1$,
and apply Theorem~\ref{t.AB_entries_comp}:
\begin{eqnarray*}
\sum_{\mu} |(AM_n)_{\la,\mu}|
&=& \sum_{\mu:\, \mu \text{ refines } \la} \prod_i \mu_{\rm init}(\la_i) 
= \prod_i \sum_{\mu^{(i)}:\, \mu^{(i)} \text{ refines } \la_i} \mu^{(i)}_{\rm init}(\la_i) \\
&=& \prod_i \sum_{\mu^{(i)}:\, \mu^{(i)} \text{ refines } \la_i} \mu^{(i)}_1 
= \prod_i \, \sum_{k=1}^{\la_i}
\sum_{\mu^{(i)}:\, \mu^{(i)} \text{ refines } \la_i \atop \mu^{(i)}_1 = k} k \\
&=& \prod_i \left( \la_i + \sum_{k=1}^{\la_i - 1} 2^{\la_i - k - 1} k \right) 
= \prod_i \left( 2^{\la_i} - 1 \right),
\end{eqnarray*}
where the last equality follows from the elementary identity
\[
\sum_{k=1}^{m - 1} 2^{m - k - 1} k =  2^m - m - 1 \qquad (m \ge 1).
\]
For $BM_n$ each product $\prod_i$ should be replaced by $\prod_i '$,
and for $HM_n$ one has
\[
\sum_J |(HM_n)_{I,J}|
= \sum_{J:\, J \subseteq I} 2^{|J|}
= \sum_{j = 0}^{|I|} {|I| \choose j} 2^{j}
= 3^{|I|}.
\]
\end{proof}

\begin{theorem}\label{t.col_sums}
{\rm (Column sums of $AM_n$, $BM_n$ and diagonal entries of $A_n^2$, $B_n^2$)}\\
Let $\mu$ be a composition of $n+1$,
and let $J = I(\mu)$ be the corresponding subset  of $[n]$.
Let $\mu^*$ be the composition of $n$ obtained from $\mu$ by reducing its first component by $1$,
without changing the other components: $\mu_1^* = \mu_1 - 1$, $\mu_i^* = \mu_i$ $(\forall i > 1)$.
\begin{itemize}
\item[$(i)$]
The sum of absolute values (also: absolute value of the sum)
of all the entries in column $J$ of $AM_n$ is equal to
the diagonal entry $(A_n^2)_{J,J}$, which in turn is equal to
\[
\prod_i (\mu_i^* + 1).
\]
\item[$(ii)$]
The sum of absolute values (also: absolute value of the sum)
of all the entries in column $J$ of $BM_n$ is equal to
the diagonal entry $(B_n^2)_{J,J}$, which in turn is equal to
\[
{\prod_i}' (\mu_i^* + 1),
\]
where $\prod_i '$ is a product over all values of $i$ except the last.
\item[$(iii)$]
For comparison, the sum of absolute values
of all the entries in column $J$ of $HM_n$ is equal to
the diagonal entry $(H_n^2)_{J,J}$, which in turn is equal to the constant $2^n$.
\end{itemize}
\end{theorem}

\begin{proof}
The recursions for $A_n^2$ and $B_n^2$ are
\[
A_n^2 = \left(\begin{array}{cc}
2 A_{n-1}^2 & A_{n-1}(A_{n-1} - B_{n-1}) \\
(A_{n-1} - B_{n-1})A_{n-1} &  A_{n-1}^2 +  B_{n-1}^2
\end{array}\right)
\qquad(n \ge 1)
\]
and
\[
B_n^2 = \left(\begin{array}{cc}
A_{n-1}^2 & A_{n-1}(A_{n-1} - B_{n-1}) \\
0 & B_{n-1}^2
\end{array}\right)
\qquad(n \ge 1),
\]
with $A_0^2 = B_0^2 = (1)$. Denoting
$\alpha_n(J) := (A_n^2)_{J,J}$, $\beta_n(J) := (B_n^2)_{J,J}$ and
$J' := J \setminus \{n\}$, it follows that
\begin{equation}\label{eq.alpha}
\alpha_n(J) =
\begin{cases}
2 \alpha_{n-1}(J'), & \text{if $n \not\in J$;} \\
\alpha_{n-1}(J') + \beta_{n-1}(J'), & \text{otherwise}
\end{cases}
\end{equation}
and
\begin{equation}\label{eq.beta}
\beta_n(J) =
\begin{cases}
\alpha_{n-1}(J'), & \text{if $n \not\in J$;} \\
\beta_{n-1}(J'), & \text{otherwise,}
\end{cases}
\end{equation}
with $\alpha_0(\emptyset) = \beta_0(\emptyset) = 1$.

A short look at recursions (\ref{eq.AMn}) and (\ref{eq.BMn})
(together with Theorem~\ref{t.AB_entries}$(ii)$) shows that
recursions (\ref{eq.alpha}) and (\ref{eq.beta}) also hold if
$\alpha_n(J)$ and $\beta_n(J)$ denote the sum of absolute values of
all the entries in column $J$ of $AM_n$ and of $BM_n$, respectively.
These recursions also hold if $\alpha_n(J)$ and $\beta_n(J)$ stand for
the explicit product formulas in the theorem,
since if $J = I(\mu)$ and $J' = J \setminus \{n\} = I(\mu')$ then
$n \not\in J$ means that $\mu$ is obtained from $\mu'$ by appending a new component of size $1$,
while $n \in J$ means that $\mu$ is obtained from $\mu'$ by increasing the last component by $1$.
\end{proof}

Comparing Theorem~\ref{t.col_sums} with Theorem~\ref{t.GilAlon} gives a surprising conclusion.

\begin{corollary}\label{t.ev_eq_diag}
The multiset of eigenvalues, counted by algebraic multiplicity, of $A_n^2$ (or $B_n^2$)
is equal to the multiset of diagonal entries of this matrix.
\end{corollary}

This is remarkable since, apparently, for $n \ge 3$ the matrices $A_n^2$ are
not even diagonalizable!

\medskip
\subsection{Inverse matrix entries}\ 
\medskip

We would like to have explicit expressions for the entries of $A_n^{-1}$,
for use in Section~\ref{section:applications}.
This turns out to be difficult to do directly, and we shall compute, as an intermediate step,
the entries of $AM_n^{-1}$. Note that $A_n^{-1} = M_n \cdot AM_n^{-1}$.

\begin{example}
\[
A_3^{-1} = \left(\begin{array}{cccccccc}
1/24 & 1/24 & 1/12 & 1/12 & 1/8 & 1/8 & 1/4 & 1/4 \\
1/8 & -1/24 & 1/12 & -1/12 & 5/24 & -1/8 & 1/12 & -1/4 \\
5/24 & 5/24 & -1/12 & -1/12 & 1/8 & 1/8 & -1/4 & -1/4 \\
1/8 & -5/24 & -1/12 & 1/12 & 1/24 & -1/8 & -1/12 & 1/4 \\
1/8 & 1/8 & 1/4 & 1/4 & -1/8 & -1/8 & -1/4 & -1/4 \\
5/24 & -1/8 & 1/12 & -1/4 & -5/24 & 1/8 & -1/12 & 1/4 \\
1/8 & 1/8 & -1/4 & -1/4 & -1/8 & -1/8 & 1/4 & 1/4 \\
1/24 & -1/8 & -1/12 & 1/4 & -1/24 & 1/8 & 1/12 & -1/4
\end{array}\right)
\]
\[
AM_3^{-1} = \left(\begin{array}{cccccccc}
1 & 0 & 0 & 0 & 0 & 0 & 0 & 0 \\
1/2 & -1/2 & 0 & 0 & 0 & 0 & 0 & 0 \\
1/2 & 0 & -1/2 & 0 & 0 & 0 & 0 & 0 \\
1/6 & -1/3 & -1/6 & 1/3 & 0 & 0 & 0 & 0 \\
1/2 & 0 & 0 & 0 & -1/2 & 0 & 0 & 0 \\
1/4 & -1/4 & 0 & 0 & -1/4 & 1/4 & 0 & 0 \\
1/6 & 0 & -1/3 & 0 & -1/6 & 0 & 1/3 & 0 \\
1/24 & -1/8 & -1/12 & 1/4 & -1/24 & 1/8 & 1/12 & -1/4
\end{array}\right)
\]
\end{example}

We shall attempt an inductive computation of $AM_n^{-1}$. The
recursion formulas~(\ref{eq.AMn}) and~(\ref{eq.BMn}) yield
corresponding recursions for  the inverse matrices:
\[
AM_n^{-1} = \left(\begin{array}{cc}
AM_{n-1}^{-1} & 0 \\
(AM_{n-1} + BM_{n-1})^{-1} & -(AM_{n-1} + BM_{n-1})^{-1}
\end{array}\right)
\qquad(n \ge 1)
\]
and
\[
BM_n^{-1} = \left(\begin{array}{cc}
AM_{n-1}^{-1} & 0 \\
0 & - BM_{n-1}^{-1}
\end{array}\right)
\qquad(n \ge 1),
\]
with $AM_0^{-1} = BM_0^{-1} = (1)$;
however, the recursion for $AM_n^{-1}$ involves the inverse of a new matrix, $AM_{n-1} + BM_{n-1}$,
which in turn involves the inverse of $AM_{n-2} + 2 BM_{n-2}$, and so forth.
We are thus led to consider a more general situation.
\begin{defn}\label{def.Mnx}
For any real number $x$ let
\[
AM_n(x) := x AM_n + (1-x) BM_n.
\]
\end{defn}

In particular, $AM_n(0) = BM_n$ and $AM_n(1) = AM_n$.

\begin{theorem}
For each $n \ge 0$ and $x > 0$,
\[
AM_n^{-1}(x)_{I,J} \ne 0 \iff J \subseteq I
\]
and, for $J \subseteq I$,
\[
AM_n^{-1}(x)_{I,J} = (-1)^{|J|} \prod_{i \in I} \frac{d_{I, J, x}(i)}{e_{I, J, x}(i)},
\]
where $I_1, \ldots, I_t$ are the runs of $I$ and, for $i \in I_k$:
\begin{itemize}
\item[$(i)$]
If $n \not\in I_k$ then
\[
d_{I, J, x}(i) :=
\begin{cases}
\max(I_k) - i + 1, & \text{if $i \in J$;}\\
1, & \text{otherwise}
\end{cases}
\]
and
\[
e_{I, J, x}(i) := \max(I_k) - i + 2.
\]
\item[$(ii)$]
If $n \in I_k$ (and thus necessarily $k = t$) then
\[
d_{I, J, x}(i) :=
\begin{cases}
(\max(I_k) - i) \cdot x + 1, & \text{if $i \in J$;}\\
x, & \text{otherwise}
\end{cases}
\]
and
\[
e_{I, J, x}(i) := (\max(I_k) - i + 1) \cdot x + 1.
\]
\end{itemize}
\end{theorem}


\begin{proof}
Let $x > 0$. By Definition~\ref{def.Mnx} and recursion formulas~(\ref{eq.AMn}) and~(\ref{eq.BMn}),
\begin{align*}
AM_n(x)
&= \left(\begin{array}{cc}
AM_{n-1} & 0 \\
x AM_{n-1} & -(x AM_{n-1} + BM_{n-1})
\end{array}\right) \\
&= \left(\begin{array}{cc}
AM_{n-1}(1) & 0 \\
x AM_{n-1}(1) & -(1+x) AM_{n-1}\left(\frac{x}{1+x}\right)
\end{array}\right)
\qquad(n \ge 1)
\end{align*}
with $AM_0(x) = (1)$.
Invertibility of $AM_{n-1}(x)$ for all $x > 0$ clearly implies the
invertibility of $AM_{n}(x)$ for all $x > 0$.
The inverse satisfies
\begin{equation}\label{e.Mnx_inverse_rec}
AM_n^{-1}(x) = \left(\begin{array}{cc}
AM_{n-1}^{-1}(1) & 0 \\
\frac{x}{1+x} AM_{n-1}^{-1}\left(\frac{x}{1+x}\right) & \frac{-1}{1+x} AM_{n-1}^{-1}\left(\frac{x}{1+x}\right)
\end{array}\right)
\qquad(n \ge 1)
\end{equation}
with $AM_0^{-1}(x) = (1)$, for all $x > 0$.

Recursion~(\ref{e.Mnx_inverse_rec}) shows that, indeed, for $x > 0$:
$AM_n^{-1}(x)_{I,J} \ne 0 \iff J \subseteq I$,
and that the sign of this entry is $(-1)^{|J|}$.

Regarding the absolute value of this entry (for $J \subseteq I$),
assume by induction that the prescribed formula holds for $AM_{n-1}^{-1}(x)$, $\forall x > 0$.

If $n \not\in I$ then also $n \not\in J$, and
clearly $AM_{n}^{-1}(x)_{I,J} = AM_{n-1}^{-1}(1)_{I,J}$ satisfies the required formula.

If $n \in I$, let $I' := I \setminus \{n\}$, $J' := J \setminus \{n\}$ and $x' := \frac{x}{1+x}$.
The assumed formula for $AM_{n-1}^{-1}(x')_{I',J'}$ and
the claimed formula for $AM_{n}^{-1}(x)_{I,J}$ have exactly the same factors for all $i \not\in I_t$,
so we need only consider $i \in I_t$.

If $|I_t| = 1$ (i.e., $n-1 \not\in I$) then there is nothing else in $AM_{n-1}^{-1}(x')_{I',J'}$,
but according to~(\ref{e.Mnx_inverse_rec}) there is an extra factor
$\frac{1}{1+x}$ or $\frac{x}{1+x}$ in $AM_{n}^{-1}(x)_{I,J}$ (depending on whether or not $n \in J$),
and this is exactly the missing $d_{I, J, x}(n)/e_{I, J, x}(n)$.

Finally, assume that $|I_t| > 1$. Again, the extra factor $\frac{1}{1+x}$ or $\frac{x}{1+x}$
is exactly $d_{I, J, x}(n)/e_{I, J, x}(n)$. The other factors in $AM_{n-1}^{-1}(x')_{I',J'}$,
corresponding to $i \in I'_t$, are (if $i \in J$)
\[
\frac{d_{I', J', x'}(i)}{e_{I', J', x'}(i)} =
\frac{(n-1-i)x'+1}{(n-i)x'+1} =
\frac{(n-1-i)x+1+x}{(n-i)x+1+x}=
\frac{d_{I, J, x}(i)}{e_{I, J, x}(i)}
\]
or (if $i \not\in J$)
\[
\frac{d_{I', J', x'}(i)}{e_{I', J', x'}(i)} =
\frac{x'}{(n-i)x'+1} =
\frac{x}{(n-i)x+1+x}=
\frac{d_{I, J, x}(i)}{e_{I, J, x}(i)},
\]
exactly as claimed for $AM_{n}^{-1}(x)_{I,J}$.
\end{proof}

We are especially interested, of course, in the special case $x = 1$.
\begin{corollary}{\rm (Entries of $AM_n^{-1}$)}\label{t.AM_inverse}\\
For each $n \ge 0$
\[
(AM_n^{-1})_{I,J} \ne 0 \iff J\subseteq I
\]
and, for $J\subseteq I$,
\[
(AM_n^{-1})_{I,J} = (-1)^{|J|} \prod_{i \in I} \frac{d_{I, J}(i)}{e_{I, J}(i)},
\]
where $I_1, \ldots, I_t$ are the runs of $I$ and, for $i \in I_k$:
\[
d_{I, J}(i) :=
\begin{cases}
\max(I_k) - i + 1, & \text{if $i \in J$;}\\
1, & \text{otherwise}
\end{cases}
\]
and
\[
e_{I, J}(i) := \max(I_k) - i + 2.
\]
Equivalently, for $J\subseteq I$,
\[
(AM_n^{-1})_{I,J} = (-1)^{|J|} \prod_{k=1}^{t} \frac{1}{(|I_k|+1)!} \prod_{i \in I_k \cap J} (\max(I_k) - i + 1).
\]
\end{corollary}

Note that the denominator $\prod_{k=1}^{t} (|I_k| + 1)!$
is the cardinality of the parabolic subgroup $\langle I \rangle$ of $S_{n+1}$
generated by the simple reflections $\{s_i\,:\, i \in I\}$.

%
%

\begin{corollary}\label{t.rows_of_inverse}\
\begin{itemize}
\item[$(i)$]
Each nonzero entry of $AM_n^{-1}$ is the inverse of an integer.
\item[$(ii)$]
In each row of $AM_n^{-1}$, the sum of absolute values of all the entries is $1$.
\item[$(iii)$]
In each row $I$ of $AM_n^{-1}$,
the first entry
\[
(AM_n^{-1})_{I, \emptyset} = \prod_{k=1}^{t} \frac{1}{(|I_k|+1)!}
\]
divides all the other nonzero entries and the diagonal entry
\[
(AM_n^{-1})_{I, I} = (-1)^{|I|} \prod_{k=1}^{t} \frac{1}{|I_k|+1}
\]
is divisible by all the other nonzero entries, where
a rational number $r$ is said to {\em divide} a rational number $s$ if the quotient $s/r$ is an integer.
\end{itemize}
\end{corollary}


\section{Fine sets revisited}\label{section:applications}

\medskip
\subsection{Fine sets and matrices}\label{subsection:fine_sets_matrices}\ 
\medskip

The mapping $\mu \mapsto I(\mu)$ (see
Subsection~\ref{subsection:prelim_compositions}) is a bijection
between the set of all compositions of $n$ and the set $P_{n-1}$
of all subsets of $[n-1]$.
For a subset $I = \{i_1, \ldots, i_k\} \subseteq [n-1]$ with $i_1 < \ldots < i_k$
let $c_I$ be the product $s_{i_1} \cdots s_{i_k} \in S_n$.
This is a Coxeter element in the parabolic subgroup generated by $\{s_i : i\in I\}$,
and its cycle type is (the partition corresponding to) the composition $\mu$,
where $I = I(\mu)$.


For an $S_n$-representation $\rho$,
let $x^\rho$ be the vector with entries $\chi^\rho(c_I)$, where
the subsets $I \in P_{n-1}$ are ordered anti-lexicographically as
in Definition~\ref{def.Pn}.

For a set of combinatorial objects $\BBB$ with a descent map
$\Des: \BBB \to P_{n-1}$, let $v^\BBB$
be the vector with entries
\[
v^\BBB_J:=|\{b\in \BBB : \Des(b)=J\}| \qquad(\forall J \in P_{n-1}),
\]
ordered anti-lexicographically as in Definition~\ref{def.Pn}.

Recall the concept of fine set (Definition~\ref{d.fine}) and the
matrices $A_n$ (Definition~\ref{d.AB_recursion}).

\begin{observation}\label{t.fine_matrix}
In the above notation,
$\BBB$ is a fine set for $\rho$ if and only if
\[
x^\rho = A_{n-1}  v^\BBB,
\]
where
$x^\rho$ and $v^\BBB$ are written as column vectors.
\end{observation}

\begin{proof}
Follows from Definition~\ref{d.fine} together with
Lemma~\ref{t.AB_explicit}$(i)$.
\end{proof}

\medskip
\subsection{Distribution of descent sets}\label{subsection:descent_sets}\ 
\medskip

We are now ready to state our main application.

\begin{theorem}\label{t.main}
If $\BBB$ is a fine set for an $S_n$-representation $\rho$ (with
respect to a map $\Des$) then the character values of $\rho$
uniquely determine the distribution of descent sets over $\BBB$.
Specifically, for every $I\subseteq [n-1]$, the number of elements
in $\BBB$ whose descent set contains $I$ satisfies
\begin{eqnarray*}
|\{b\in \BBB : \Des(b) \supseteq I\}|
&=& \sum_J (AM_{n-1}^{-1})_{I,J} \chi^\rho(c_J) \\
&=& \frac{1}{|\langle I \rangle|}
\sum_{J:\,J\subseteq I} (-1)^{|J|}\chi^\rho(c_J)
\prod_{k=1}^{t} \prod_{i \in I_k \cap J} (\max(I_k) - i + 1),
\end{eqnarray*}
where $I_1, \ldots, I_t$ are the runs in $I$,
$|\langle I \rangle| = \prod_k (|I_k|+1)!$ is the cardinality of the parabolic subgroup of $S_n$
generated by $\{s_i : i\in I\}$,
and $c_I$ is any Coxeter element in this subgroup.
\end{theorem}

\begin{proof}
By Theorem~\ref{t.An-determinant}, $A_{n-1}$ is an invertible
matrix. Combining this with Observation~\ref{t.fine_matrix} proves
that $x^\rho$ uniquely determines $v^\BBB$.

Writing
the equality $x^\rho = A_{n-1}  v^\BBB$ in the equivalent form
\[
Z_{n-1} v^\BBB = AM_{n-1}^{-1} x^\rho,
\]
the explicit formula now follows from Corollary~\ref{t.AM_inverse}.
\end{proof}





Using the Inclusion-Exclusion Principle (namely, multiplying by $M_{n-1}$ on the left)
leads to another version of the explicit formula.

\begin{corollary}\label{t.main_cor}
Let $\BBB$ be a fine set for an $S_n$-representation $\rho$ (with respect to a map $\Des$).
Then, for every $D \subseteq [n-1]$, the number of elements in $\BBB$ with
descent set exactly $D$ satisfies
\[
|\{b\in \BBB : \Des(b) = D\}|
= \sum_J (A_{n-1}^{-1})_{D,J} \chi^\rho(c_J)
\]
\[
=
\sum_{J} (-1)^{|J|}\chi^\rho(c_J)
\sum_{I:\, I \supseteq D \cap J} \frac{(-1)^{|I \setminus D|}}{|\langle I \rangle|}
\prod_{k=1}^{t} \prod_{i \in I_k \cap J} (\max(I_k) - i + 1),
\]
with notation is as in Theorem~\ref{t.main}.
\end{corollary}

\section{Permutation statistics and character theory - an application}\label{section:applications2}


By Theorem~\ref{t.main}, certain statements concerning permutation statistics
have equivalent formulations in character theory. In particular
(using the language of $S_n$-modules instead of $S_n$-representations),

\begin{corollary}\label{t.ps-rt}
Given two $S_n$-modules with fine sets, the modules are isomorphic
if and only if their fine sets have the same descent set
distribution.
\end{corollary}



Here is a distinguished example.

\begin{defn}
Let $\BBB$ be a fine set with descent map $\Des$.
The {\em major index} of an element $b \in \BBB$ is
\[
\maj(b):=\sum_{i \in \Des(b)} i.
\]
\end{defn}
For a subset $I\subseteq [n-1]$ denote $\x^I:=\prod_{i\in I} x_i$,
where $x_i$ are indeterminates.
The following is a fundamental theorem on permutation statistics,
stated in its usual generating function form.


\begin{theorem}\label{t.FS-thm}
{\rm (Foata-Sch\"utzenberger)}~\cite[Theorem 1]{FS79}
\[
\sum_{\pi \in S_n} \x^{\Des(\pi)} q^{\ell(\pi^{-1})} =
\sum_{\pi \in S_n} \x^{\Des(\pi)} q^{\maj(\pi^{-1})}.
\]
\end{theorem}

Note that $\ell(\pi^{-1})$ is used for symmetry,
but actually $\ell(\pi^{-1}) = \ell(\pi)$.
See also~\cite{GG}.

\medskip

For $0\le k\le {n\choose 2}$ denote by $R_k$
the $k$-th homogeneous component of the coinvariant algebra of the
symmetric group $S_n$. The following is a classical theorem in
invariant theory.

\begin{theorem}\label{t.St-thm}
{\rm (Lusztig-Stanley)}~\cite[Prop.~4.11]{St79}
For a partition $\la$ of $n$ denote by $m_{k,\la}$ the number of
standard Young tableaux of shape $\la$ with major index $k$.
Then
\[
R_k \cong \bigoplus_{\la \vdash n} m_{k,\la} S^\la,
\]
where the sum is over all partitions of $n$ and $S^\la$
denotes the irreducible $S_n$-module indexed by $\la$.
\end{theorem}

Recall the concept of Knuth class (Definition~\ref{d.Knuth_class}).

\begin{claim}\label{t.fine_Knuth} 
Any Knuth class of shape $\la$ is a fine set for the 
irreducible module $S^\la$.
\end{claim}

\begin{proof}
By the Robinson-Schensted correspondence (see
Subsection~\ref{subsection:perms_SYT}), for each standard Young
tableau $T$ the permutations in the Knuth class $\CCC_T$ are in
bijection with the standard Young tableaux $Q$ of the same shape
$\la$ as $T$. By Fact~\ref{t.Des_RS}, this bijection preserves
descent set; and therefore, by
Proposition~\ref{t.fine-examples}$(i)$, $\CCC_T$ is a fine set for
the irreducible module $S^\la$.
\end{proof}


\begin{claim}
The Foata-Sch\"utzenberger Theorem is equivalent to the Lusztig-Stanley Theorem.
\end{claim}

%

\begin{proof}
By Proposition~\ref{t.fine-examples}$(ii)$, the set
of permutations 
$L_k=\{\pi \in S_n:\ \ell(\pi) = k\}$ is a fine set for $R_k$.

On the other hand, by Corollary~\ref{t.inverse_des_class}, the set of
permutations $\BBB_k := \{\pi \in S_n:\ \maj(\pi^{-1}) = k\}$
is a disjoint union of Knuth classes.
In fact, if $\pi \in S_n$ corresponds under the Robinson-Schensted correspondence
to a pair $(P_\pi, Q_\pi)$ of standard Young tableaux, then
$\maj(\pi^{-1}) = \maj(P_\pi)$.
Thus $B_k$ contains, for each $\la \vdash n$, exactly
$m_{k,\la}$ Knuth classes of shape $\la$
(in the notation of Theorem~\ref{t.St-thm}).
By Claim~\ref{t.fine_Knuth}, this
implies that $\BBB_k$ is a fine set for the $S_n$-module
$R'_k := \bigoplus_{\la \vdash n} m_{k,\la} S^\la$.

Therefore, by Corollary~\ref{t.ps-rt},
$R_k\cong R'_k$ (which is the claim of Theorem~\ref{t.St-thm}) holds
if and only if $L_k$ and $\BBB_k$ have the same descent set
distribution (which is the claim of Theorem~\ref{t.FS-thm}).

\end{proof}

\begin{remark} A combinatorial proof of the Lusztig-Stanley Theorem
as an application of the Foata-Sch\"utzenberger Theorem appears
in~\cite{Ro-Schubert}. The opposite implication seems to be new.
\end{remark}

\section{Fine sets in context}\label{section:fine_sets}


\medskip
\subsection{Equidistribution and quasi-symmetric functions}\label{subsection:quasi}\
\medskip

Another useful criterion is the following. For a partition
$\la \vdash n$ let $\SYT(\la)$ be the set of standard Young tableaux
of shape $\la$.

\begin{proposition}\label{t.criterion1}
Let $\BBB$ be a set of combinatorial objects, equipped with a map
$\Des: \BBB \to P_{n-1}$. Then $\BBB$ is a fine set if and only if
there exists a collection $(c_\la)_{\la \vdash n}$ 
of nonnegative integers such that
\[
\sum_{b \in \BBB} {\bf x}^{\Des(b)} =
\sum_{\la \vdash n} c_\la \sum_{T \in SYT(\la)} {\bf x}^{\Des(T)},
\]
where $\Des(T)$ on the RHS is as in Definition~\ref{d.Des_SYT}.
\end{proposition}

\begin{proof}
In the notation of Observation~\ref{t.fine_matrix},
$\BBB$ is a fine set for a representation $\rho$ if and only if
\[
x^\rho = A_{n-1} v^\BBB.
\]
On the other hand,
$x^\rho$ is a linear combination (with nonnegative integer coefficients) of
the vectors corresponding to irreducible representations.
Since $\SYT(\la)$ is a fine set for the irreducible module $S^\la$
(by Proposition~\ref{t.fine-examples}$(i)$),
the invertibility of $A_{n-1}$ implies that $\BBB$ is a fine set if and only if
\[
v^\BBB = \sum_{\la \vdash n} c_\la v^{\SYT(\la)}
\]
for suitable nonnegative integers $c_\la$.
By definition, this means that
\[
|\{b \in \BBB \,|\, \Des(b) = D\}| =
\sum_{\la \vdash n} c_\la \cdot |\{T \in \SYT(\la) \,|\, \Des(T) = D\}|
\qquad (\forall D \subseteq [n-1]),
\]
which is exactly the content of the stated polynomial equality.
\end{proof}






%

By~\cite[Theorem 2.1]{GR}, as reformulated in~\cite[Theorem 2.2]{R13},
conjugacy classes in the symmetric group satisfy 
this criterion. It follows that any subset of the symmetric group which is closed
under conjugation is a fine set.

\medskip

Another example satisfying this criterion has been given recently.
\begin{defn}\label{d.arc_perm}
A permutation $\pi\in S_n$ is called an {\em arc permutation} if,
for every $1\le k\le n$, the set $\{\pi(1),\dots,\pi(k)\}$ forms
an interval in the cyclic group $\ZZ_n$ (where $n$ is identified with 0).
\end{defn}
By~\cite[Theorem 5]{ER}, the set of arc permutations in $S_n$
satisfies the criterion of Proposition~\ref{t.criterion1} when $n \ge 4$.
For $n \le 3$ all permutations in $S_n$ are arc permutations
and thus, by Corollary~\ref{t.c_regular}, they form a fine set.
Therefore the set of arc permutations in $S_n$ is a fine set for all $n$.




\medskip

It is convenient to reformulate Proposition~\ref{t.criterion1}
in the language of quasi-symmetric functions.


Schur functions, indexed by partitions, form a distinguished basis for
the ring
of symmetric functions; 
see, e.g., \cite[Corollary 7.10.6]{Stanley_ECII}.
A symmetric function 
is {\em Schur positive} if all the coefficients in its expansion in
the basis of Schur functions are nonnegative. The problem of
determining whether a given symmetric function is Schur positive
is a major problems in contemporary algebraic combinatorics~\cite{Stanley_problems}.


Let $\BBB$ a set of combinatorial objects, equipped with a descent
map $\Des: \BBB \to P_{n-1}$.
For each subset $D \subseteq [n-1]$ define 
\[
F_D(\x) := \sum\limits_{i_1\le i_2 \le \ldots \le i_n \atop
{i_j < i_{j+1} \text{ if } j \in D}} x_{i_1} x_{i_2} \cdots x_{i_n},
\]
and let
\[
Q_{\BBB} := \sum\limits_{b\in \BBB} F_{\,\Des(b)}.
\]
Proposition~\ref{t.criterion1} may now be reformulated as follows.

\begin{proposition}\label{t.criterion_q}
The set $\BBB$ is a fine set if and only if the quasi-symmetric
function $Q_\BBB$ is symmetric and Schur positive.
\end{proposition}

\begin{proof}
The combinatorial definition of a Schur function $s_\la$
(see, e.g., \cite[Definition 7.10.1]{Stanley_ECII})
implies that
\[
s_\la = Q_{\SYT(\la)}.
\]
By definition,
\[
Q_\BBB
= \sum_{b \in \BBB} F_{\,\Des(b)}
= \sum_{D \subseteq [n-1]} |\{b \in \BBB \,|\, \Des(b) = D\}| \cdot F_D
\]
and similarly
\[
s_\la = Q_{\SYT(\la)}
= \sum_{T \in \SYT(\la)} F_{\,\Des(T)}
= \sum_{D \subseteq [n-1]} |\{T \in \SYT(\la) \,|\, \Des(T) = D\}| \cdot F_D.
\]
As in the proof of Proposition~\ref{t.criterion1}, $\BBB$ is a fine set if and only if
there exist nonnegative integers $(c_\la)_{\la \vdash n}$ such that
\[
|\{b \in \BBB \,|\, \Des(b) = D\}| =
\sum_{\la \vdash n} c_\la \cdot |\{T \in \SYT(\la) \,|\, \Des(T) = D\}|
\quad (\forall D \subseteq [n-1]).
\]
This equality implies that
\[
Q_\BBB = \sum_{\la \vdash n} c_\la s_\la
\]
and, since $\{F_D \,|\, D \subseteq [n-1]\}$ are linearly independent,
the converse also holds.
Thus $\BBB$ is a fine set if and only if $Q_\BBB$ is symmetric and Schur positive.
\end{proof}

The challenging problem of characterizing the fine subsets of the symmetric group
may now be rephrased.

\begin{question} For which $A\subseteq S_n$ is $F_A$ symmetric
and Schur positive?
\end{question}

We conclude with a list of the known fine subsets of the symmetric group.
The first two examples appeared in~\cite[Theorem 5.5]{GR}.
Recall the relevant definitions from
Subsection~\ref{subsection:perms_SYT}. 

%

\begin{proposition}
The following subsets of $S_n$ are fine sets:
\begin{itemize}
\item[$(i)$]
Knuth classes and their unions (including
inverse descent classes and the set of $321$-avoiding permutations).
\item[$(ii)$]
Conjugacy classes and their unions.
\item[$(iii)$]
The set of permutations of fixed Coxeter length.
\item[$(iv)$]
The set of arc permutations.
\end{itemize}
\end{proposition}

\begin{proof}\
\begin{itemize}
\item[$(i)$]
By Claim~\ref{t.fine_Knuth}, any Knuth class is a fine set.
By Corollaries~\ref{t.inverse_des_class} and~\ref{t.321_avoiding},
inverse descent classes and the set of all $321$-avoiding permutations
are disjoint unions of Knuth classes.
\item[$(ii)$]
By the remarks following Proposition~\ref{t.criterion1}.
\item[$(iii)$]
By Proposition~\ref{t.fine-examples}$(ii)$.
\item[$(iv)$]
By the remarks following Definition~\ref{d.arc_perm}.
\end{itemize}
\end{proof}

\medskip
\subsection{Distinguished bases}\label{subsection:canonical}\
\medskip

In their seminal paper~\cite{KL}, Kazhdan and Lusztig constructed
complex representations of a Coxeter group $W$ by partitioning the
group into cells and introducing a distinguished basis,
indexed by the group elements.  The group algebra, and more
generally the Iwahori-Hecke algebra, is then decomposed into
a direct sum of subspaces, each spanned by the (normalized) basis elements
indexed by the elements of one cell.
For the symmetric group, the (left) cells are the Knuth classes and
the resulting representations are exactly the irreducible ones.


\begin{theorem}\label{t.KL1}(Kazhdan-Lusztig Theorem)\ \\
Let $\CCC$ be a Knuth class of shape $\lambda$ in $S_n$. Then
\begin{itemize}
\item[$(i)$]
$\CCC$ is a left cell in $S_n$.
\item[$(ii)$]
There exist coefficients $a_i(b,b')$, defined for
$1 \le i \le n-1$ and $b, b' \in \CCC$, such that
the following action of the simple reflections $s_i$
determines an $S_n$-representation:
\[
s_i (C_b) =
\begin{cases}
-C_b, &\text{\rm if } i \in \Des (b); \\
C_b + \sum_{b' \in \BBB \text{ \rm s.t.\ } i \in\Des(b')}
a_i(b,b') C_{b'}, &\text{\rm otherwise.}
\end{cases}
\]
\item[$(iii)$]
The resulting module is irreducible and isomorphic to $S^\la$.
\end{itemize}
\end{theorem}

For the first part see~\cite[Fact 8]{GM} and~\cite[\S 6.4]{BB}.
The second part is a special case of~\cite[(2.3.a)-(2.3.d)]{KL},
see also~\cite[\S 7.4]{Humphreys} and~\cite[(6.4)]{BB}. For the
last part see~\cite[Theorem 1.4]{KL}.

\medskip

A very similar phenomenon occurs in the study of the homogeneous
components of the coinvariant algebra of classical Weyl groups.
Schubert polynomials, indexed by elements of fixed Coxeter length
$k$, form a basis for the $k$-th homogeneous component. A formula
similar to the one in Theorem~\ref{t.KL1}$(ii)$ was proved
in~\cite[Lemma 3.1]{Ro-Schubert} for the symmetric group and
in~\cite{APR11} for classical Weyl groups; 
see also~\cite[Theorem 3.14($iii$)]{BGG}.

\medskip


This terminology provides another useful characterization of fine sets.

\begin{theorem}\label{t.condition-fine}
Let $\BBB$ be a set of combinatorial objects, equipped with a map
$\Des : \BBB \to P_{n-1}$. Then $\BBB$ is a fine set for an
$S_n$-representation $\rho$ if and only if there exists a basis
$\{C_b : b\in \BBB\}$ for the corresponding representation space
such that for every $1 \le i \le n-1$ and $b, b' \in \BBB$
\begin{equation}\label{eq3.1}
\rho(s_i) (C_b) =
\begin{cases}
-C_b, &\text{\rm if } i \in \Des (b); \\
C_b + \sum_{b' \in \BBB \text{ \rm s.t.\ } i \in \Des(b')}
a_i(b,b') C_{b'}, &\text{\rm otherwise,}
\end{cases}
\end{equation}
for suitable coefficients $a_i(b,b')$.
\end{theorem}

\begin{proof}
First we prove that the existence of a basis with coefficients satisfying (\ref{eq3.1})
implies that $\BBB$ is fine.
The proof is a natural extension of the proofs of~\cite[Theorem 2]{Ro2}
and~\cite[Theorem 1]{Ro-Schubert}.
Let $\langle \cdot, \cdot \rangle$ be the inner product on $V:={\rm
span}\{C_b:\ b\in \BBB\}$ defined by
\[
\langle C_{b},C_{b'}\rangle =\delta_{b,b'} :=
\begin{cases}
1, & \text{\rm if } b = b'; \\
0, & \text{\rm otherwise.}
\end{cases}
\]

Let $S_n$ be the symmetric group on the letters $1,\ldots,n$. For
$1 \le i \le n-1$ denote $s_i := (i,i+1)$, a simple reflection
(adjacent transposition) in $S_n$. For a composition $\mu =
(\mu_1, \dots, \mu_t)$ of $n$ define $\sigma_\mu \in S_n$ by
\[
\sigma_\mu^{-1} := (1, 2, \ldots, \mu_1)(\mu_1+1, \mu_1+2, \ldots,
\mu_1+\mu_2) \cdots,
\]
a product of $t$ cycles of lengths $\mu_1, \mu_2, \ldots, \mu_t$
consisting of consecutive letters. The permutation
$\sigma_\mu^{-1}$ may be obtained from the product $s_1 s_2 \cdots
s_{n-1}$ of all simple reflections (in the usual order) by
deleting the factors $s_{\mu_1+\ldots+\mu_k}$ for all $1 \le k <
t$, namely
\[
\sigma_\mu^{-1} = \prod_{i \in I(\mu)} s_i,
\]
and $\sigma_\mu$ is the product of the same factors in the reverse
order.


By definition,
\[
\chi^\rho(\sigma_\mu) = \sum\limits_{b\in \BBB} \langle
\rho(\sigma_\mu)(C_b), C_b \rangle.
\]
From now on we omit $\rho$, for ease of notation.

The permutation $\sigma_\mu$ has cycle type $\mu$. Thus, in order
to prove that $\BBB$ is a fine set, it suffices to prove that for
every composition $\mu$ of $n$ and every $b\in \BBB$
\begin{equation}\label{eq.key}
\langle \sigma_\mu(C_b), C_b \rangle =
\begin{cases}
(-1)^{|\Des(b)\cap I(\mu)|}, &\text{\rm if $b$ is $\mu$-unimodal};\\
0, &\text{\rm otherwise.}
\end{cases}
\end{equation}


Assume first that $b \in \BBB$ is not $\mu$-unimodal. Then there
exists an index $i$ such that $i, i+1 \in I(\mu)$,
$i \not\in \Des(b)$ and $i+1 \in \Des(b)$. By (\ref{eq3.1}),
\[
i \not\in \Des(b) \then \langle s_{i}(C_{b'}), C_{b}\rangle =
\delta_{b', b} \qquad (\forall b, b' \in \BBB).
\]
It follows that, by linearity,
\begin{equation}\label{eq3.3}
i \not\in \Des(b) \then \langle s_{i}(v), C_b \rangle = \langle v,
C_b \rangle \qquad (\forall b \in \BBB,\, v \in V)
\end{equation}
and, in particular,
\[
\langle s_{i} \sigma_\mu(C_b), C_b\rangle = \langle
\sigma_\mu(C_b), C_b\rangle.
\]
On the other hand, by (\ref{eq3.1}),
\[
i+1 \in \Des(b) \then \langle \sigma_\mu s_{i+1}(C_b), C_b \rangle
= -\langle \sigma_\mu(C_b), C_b\rangle.
\]
The braid relation $s_i s_{i+1} s_i = s_{i+1} s_i s_{i+1}$ and the
commuting relations  $s_j s_k = s_k s_j$ for $|j - k|>1$ imply
that
\[
s_{i} \sigma_\mu = \sigma_\mu s_{i+1}.
\]
The last three equalities combine to give
\[
\langle \sigma_\mu(C_b), C_b\rangle = \langle s_{i}
\sigma_\mu(C_b), C_b\rangle = \langle \sigma_\mu s_{i+1}(C_b), C_b
\rangle = -\langle \sigma_\mu(C_b), C_b\rangle,
\]
or equivalently $\langle  \sigma_\mu(C_b), C_b \rangle = 0$,
as claimed in~(\ref{eq.key}).

\medskip

It remains to compute $\langle \sigma_\mu(C_b), C_b \rangle$ when
$b$ is $\mu$-unimodal.
The $\mu$-unimodality of $b$ and the commuting relations $s_j s_k
= s_k s_j$ for $|j - k|>1$ make it possible to push all the
factors in $\sigma_\mu$ with indices in $\Des(b)$ to the right, so
that $\sigma_\mu$ may be written in the form
\[
\sigma_\mu =  s_{i_{n-t}} \cdots s_{i_{m+1}} s_{i_m} \cdots
s_{i_1},
\]
where $t$ is the number of parts in $\mu$ (so that $n-t =
|I(\mu)|$), $m:=|\Des(b)\cap I(\mu)|$, and $i_j \in \Des(b)$ if
and only if $1 \le j \le m$. Then (\ref{eq3.3}) implies that
\[
\langle \sigma_\mu(C_b), C_b \rangle = \langle s_{i_m}\cdots
s_{i_1}(C_b), C_b \rangle
\]
and (\ref{eq3.1}) finally implies that
\[
\langle \sigma_\mu(C_b), C_b \rangle = (-1)^m,
\]
as claimed in~(\ref{eq.key}), completing the proof of this
direction.


\medskip

For the opposite direction, assume that $\BBB$ is fine.
By Proposition~\ref{t.criterion1} there exists a collection
$(c_\la)_{\la \vdash n}$ of nonnegative integers such that
\[
\sum_{b \in \BBB} {\bf x}^{\Des(b)} =
\sum_{\la \vdash n} c_\la \sum_{T \in SYT(\la)} {\bf x}^{\Des(T)}.
\]
We can therefore partition $\BBB$ into disjoint subsets, each corresponding to
a certain $\SYT(\la)$ via a $\Des$-preserving bijection.
By Claim~\ref{t.fine_Knuth}, each $\SYT(\la)$ can be replaced by
a corresponding Knuth class of permutations which, by Theorem~\ref{t.KL1},
carries a linear action of the required form~(\ref{eq3.1}).


\end{proof}




Since $S_n$ embeds naturally into
classical Weyl groups of rank $n$, it follows that Kazhdan-Lusztig
cells, as well as subsets of elements of fixed Coxeter length in
these groups, are fine sets for the $S_n$-action on the group algebra.


\begin{thebibliography}{99}

\bibitem{APR1}
R.\ M.\ Adin, A.\ Postnikov, and Y.\ Roichman,
{\it Hecke algebra actions  on the coinvariant algebra},
J.\ Algebra~{\bf 233} (2000), 594--613.

\bibitem{APR11}
R.\ M.\ Adin, A.\ Postnikov, and Y.\ Roichman,
{\it On characters of Weyl groups},
Discrete Math.~{\bf 226} (2001), 355--358.

\bibitem{APR2}
R.\ M.\ Adin, A.\ Postnikov, and Y.\ Roichman,
{\it Combinatorial Gelfand models},
J.\ Algebra~{\bf 320} (2008), 1311--1325.


\bibitem{Alon}
G.\ Alon,
{\it Eigenvalues of the Adin-Roichman matrices}, 
Linear Algebra Appl.~{\bf 450} (2014), 280--292.

\bibitem{BGG}
I.\ N.\ Bernstein, I.\ M.\ Gelfand, S.\ I.\ Gelfand,
{\it Schubert cells and cohomology of the spaces $G/P$},
Usp.\ Mat.\ Nauk.~{\bf 28} (1973), 3--26.

\bibitem{BB}
A.\ Bj\"{o}rner and F.\ Brenti,
Combinatorics of Coxeter groups,
Graduate Texts in Mathematics 231, Springer, New York, 2005.

\bibitem{Chevalley}
C.\ Chevalley,
{\it Invariants of finite groups generated by reflections},
Amer.\ J.\ Math.~{bf 77} (1955), 778--782.

\bibitem{ER}
S.\ Elizalde and Y.\ Roichman,
{\it Arc permutations},
J.\ Algebraic Combin., to appear. 

\bibitem{FS79}
D.\ Foata and M.\ P.\ Sch\"utzenberger,
{\it Major index and inversion number of permutations},
Math.\ Nachr.~{\bf 83} (1978), 143--159.

\bibitem{GG}
A.\ M.\ Garsia and I.\ Gessel,
{\it Permutation statistics and partitions},
Adv.\ Math.~{\bf 31} (1979), 288--305.

\bibitem{GM}
A.\ M.\ Garsia and T.\ J.\ McLarnan,
{\it Relations between Young's natural and the Kazhdan-Lusztig representations of $S_n$},
Adv.\ Math.~{\bf 69} (1988), 32--92.

\bibitem{GR}  I.\ M.\ Gessel and C.\ Reutenauer,
{\em Counting permutations with given cycle structure and descent set},
J.\ Combin.\ Theory Ser.\ A~{\bf  64} (1993), 189--215.

\bibitem{Humphreys}
J.\ E.\ Humphreys,
Reflection groups and Coxeter groups,
Cambridge Studies in Advanced Mathematics~{\bf 29},
Cambridge University Press, Cambridge, 1990.

\bibitem{KL}
D.\ Kazhdan and G.\ Lusztig,
{\it Representations of Coxeter groups and Hecke algebras},
Invent.\ Math.~{\bf 53} (1979), 165--184.



\bibitem{Ra2}
A.\ Ram,
{\it An elementary proof of Roichman's rule for irreducible characters of Iwahori-Hecke algebras of type A},
in: Mathematical essays in honor of Gian-Carlo Rota, Progr.\ Math., 161, Birkh\"auser, Boston, 1998, 335--342.

\bibitem{Ro2}
Y.\ Roichman,
{\it A recursive rule for Kazhdan-Lusztig characters},
Adv.\ in Math.~{\bf129} (1997), 24--45.

\bibitem{Ro-Schubert}
Y.\ Roichman,
{\it Schubert polynomials, Kazhdan-Lusztig basis and characters},
Formal Power Series and Algebraic Combinatorics (Vienna, 1997).
Discrete Math.~{\bf 217} (2000), 353--365.


\bibitem{R13} Y.\ Roichman,
{\it A note on the number of $k$-roots in $S_n$},
preprint, 2013.

\bibitem{Rota}
G.-C.\ Rota,
{\it On the foundations of combinatorial theory. I. Theory of M\"obius functions},
Z.\ Wahrscheinlichkeitstheorie und Verw.\ Gebiete~{\bf 2} (1964), 340-–368.

\bibitem{Sagan}
B.\ E.\ Sagan, 
The symmetric group: Representations, combinatorial algorithms, and symmetric functions, Second edition, Graduate Texts in Math., no.\ 203, Springer-Verlag, New York, 2001.

\bibitem{Schensted}
C.\ Schensted, {\em Longest increasing and decreasing
subsequences}, Canad.\ J.\ Math.~{\bf 13} 1961, 179--191.

\bibitem{Sloane}
N.\ J.\ A.\ Sloane,
The On-Line Encyclopedia of Integer Sequences.

\bibitem{St79}
R.\ P.\ Stanley,
{\it Invariants of finite groups and their applications to combinatorics},
Bull.\ Amer.\ Math.\ Soc.\ (new series)~{\bf 1} (1979), 475--511.

\bibitem{Stanley_ECII}
R.\ P.\ Stanley,
Enumerative combinatorics, Vol.\ 2,
Cambridge Studies in Adv.\ Math., no.\ 62. Cambridge Univ.\ Press, Cambridge, 1999.

\bibitem{Stanley_problems}
R.\ P.\ Stanley,
{\it Positivity problems and conjectures in algebraic combinatorics},
in: Mathematics: Frontiers and Perspectives (V.\ Arnold, M.\ Atiyah, P.\ Lax, and B.\ Mazur, eds.),
American Mathematical Society, Providence, RI, 2000, pp. 295--319.

\end{thebibliography}
\end{document}